\documentclass[a4paper,12pt]{article}

\usepackage[latin1]{inputenc}
\usepackage[T1]{fontenc}
\usepackage[english]{babel}

\usepackage{mathrsfs}
\usepackage{amscd}
\usepackage{amsfonts}
\usepackage{amsmath}
\usepackage{amssymb}
\usepackage{amstext}
\usepackage{amsthm}
\usepackage{amsbsy}

\usepackage{xspace}
\usepackage[all]{xy}
\usepackage{graphicx}
\usepackage{tikz}
\usepackage{url}
\usepackage{latexsym}

\usepackage{graphicx} 

\usepackage{booktabs} 
\usepackage{array} 
\usepackage{paralist} 
\usepackage{verbatim} 
\usepackage{subfig} 

\usepackage{fancyhdr} 
\usepackage{sectsty}
\allsectionsfont{\sffamily\mdseries\upshape} 

\usepackage[nottoc,notlof,notlot]{tocbibind} 
\usepackage[titles,subfigure]{tocloft} 

\usepackage[textwidth=100pt,textsize=footnotesize,bordercolor=white,color=blue!30]{todonotes}
\usepackage{hyperref} 
\usepackage[affil-it]{authblk}

\makeatletter
\newcommand*{\rom}[1]{\expandafter\@slowromancap\romannumeral #1@}
\makeatother

\theoremstyle{definition}

\newtheorem{fact}{fact}

\newtheorem{thm}[fact]{Theorem}
\newtheorem{lemma}[fact]{Lemma}
\newtheorem{prop}[fact]{Proposition}
\newtheorem{corollary}[fact]{Corollary}
\newtheorem{defini}[fact]{Definition}

\newtheorem{question}[fact]{Question}
\newtheorem{remark}[fact]{Remark}

\title{Taming Koepke's Zoo II: Register Machines}
\author{Merlin Carl}
\affil{Institut f\"ur mathematische, naturwissenschaftliche und technische Bildung, Abteilung f\"ur Matheamtik und ihre Didaktik, Europa-Universit\"at Flensburg}

\begin{document}

\maketitle

\begin{abstract}
We study the computational strength of resetting $\alpha$-register machines, a model of transfinite computability introduced by P. Koepke in \cite{K1}. Specifically, we prove the following strengthening of a result from \cite{C}: For an exponentially closed ordinal $\alpha$, we have $L_{\alpha}\models$ZF$^{-}$ if and only if COMP$^{\text{ITRM}}_{\alpha}=L_{\alpha+1}\cap\mathfrak{P}(\alpha)$, i.e. if and only if the set of $\alpha$-ITRM-computable subsets of $\alpha$ coincides with the set of subsets of $\alpha$ in $L_{\alpha+1}$. Moreover, we show that, if $\alpha$ is exponentially closed and $L_{\alpha}\not\models$ZF$^{-}$, then COMP$^{\text{ITRM}}_{\alpha}=L_{\beta(\alpha)}\cap\mathfrak{P}(\alpha)$, where $\beta(\alpha)$ is the supremum of the $\alpha$-ITRM-clockable ordinals, which coincides with the supremum of the $\alpha$-ITRM-computable ordinals. We also determine the set of subsets of $\alpha$ computable by an $\alpha$-ITRM with time bounded below $\delta$ when $\delta>\alpha$ is an exponentially closed ordinal smaller than the supremum of the $\alpha$-ITRM-clockable ordinals.
\end{abstract}

\section{Introduction}

In \cite{KM}, Koepke and Miller introduced Infinite Time Register Machines (ITRMs) as a generalization of register machines to ordinal time, thus complementing the Infinite Time Turing Machines introduced in Hamkins and Lewis \cite{HL}. An ITRM has finitely many registers, each of which can store a single natural number. Later on, Koepke also introduced Ordinal Turing Machines (ORMs) (see e.g. \cite{ORM}), in which every register can contain an arbitrary ordinal. In \cite{K1}, he further mentions the possibility of defining register machines in which the register contents are bounded by an ordinal $\alpha$ while the computation time is bounded by a possibly different ordinal $\beta$, so-called $(\alpha,\beta)$-ITRMs. The computational objects for ITRMs were determined in \cite{K1} to be those subsets of $\omega$ contained in $L_{\omega_{\omega}^{\text{CK}}}$, and for ORMs, it was shown in \cite{ORM} that they can compute exactly the constructible sets of ordinals. Moreover, by arguments analogous to those given in Koepke and Seyfferth \cite{KS} for tape models, one can see that for exponentially closed $\alpha$, an $(\alpha,\alpha)$-ITRM computes exactly those subsets of $\alpha$ that are $\alpha$-recursive, i.e. $\Delta_{1}$ over $L_{\alpha}$. 

We recall the definitions of $\alpha$-(w)ITRMs from \cite{KM} and \cite{K}  very briefly. In the original definition of Koepke, programs for $\alpha$-(w)ITRMs are simply register machine programs as e.g. described in Cutland \cite{Cu}. An $\alpha$-(w)ITRM has finitely many registers, each of which can store a single ordinal $<\alpha$. Whe $P$ is a program that uses the register with indices $1,2,...,n$, then a $P$-computation is a sequence of $P$-configuration, i.e., elements $(l,c_{1},...,c_{n})$ of $\omega\times\alpha^{n}$, where $l$ denotes the active program line ad $c_{i}$ is the content of the $i$-th register, for $i\in\{1,2,...,n\}$. These machines operate along an ordinal time axis. At successor stages, the register machine commands are carried out as usual. At limit stages, the active program line and the register contents are obtained as the inferior limits of the sequences of earlier program lines or the earlier contents of the register in question, respectively. If that limit is $\alpha$, an $\alpha$-wITRM-computation is undefined (it ``crashes''), while in an $\alpha$-ITRM-computation, the content of such a register is simply reset to $0$. Of course, we can specify an initial configuration $c$ from which the machine will start; we write $P(c)$ for the computation of the program $P$ starting in the initial configuration $c$. When we work with a single ordinal $\iota$ as an input, we write $P(\iota)$ and understand that $\iota$ is initially written to the first register, while all other registers contain $0$. When we do not specify an input configuration, it is meant that we start in the situation where all registers contain $0$. When we talk about ``outputs'' of a computation, we refer to the configuration in the halting state; often (when, e.g., talking about the computability of functions from the ordinals to the ordinals), it is only the content $\rho$ of the first register we care about, in which case we say that $P(c)$ halts with output $\rho$. For this paper, we modify the definition of a program slightly for the sake of a smoother development: Instead of only allowing jump conditions of the form $R_{i}=R_{j}$, i.e. stating that the content of one register is equal to the content of another register, we allow arbitrary Boolean combinations of such statements as jumping conditions. This has the convenient effect that one can store a certain configuration of a program $P$ (i.e. the register contents and active program line) in some extra registers and then recognize in a single step whether or not $P$ is currently in this configuration. Clearly, this leaves the computational power of these machines untouched; one nice effect of this modification is that the proof of the speedup-theorem for $\alpha$-(w)ITRMs (see Lemma \ref{no gaps} below) avoids a lot of the trickery used in  \cite{CFKMNW} in the case of ($\omega$-)ITRMs.

In this work, we investigate the computational strength of $\alpha$-ITRMs in the case that $\alpha$ has sufficient closure properties. The lesson here is that, if $\alpha$ has strong closure properties, then the computational strength of $\alpha$-(w)ITRMs goes little beyond $L_{\alpha}$. 
The results on $\alpha$-ITRMs in this paper are refinement of results that appeared in \cite{C}; there, it was proved that, if $\alpha>\omega$ is regular in $L_{\alpha^{\omega}}$, then COMP$^{\text{ITRM}}_{\alpha}=L_{\alpha+1}\cap\mathfrak{P}(\alpha)$. Here, we exploit the proof to obtain the same result under a considerably weaker condition, thus obtaining an equivalence.

Although we will mostly be concerned with $\alpha$-ITRMs below, we point out the following result on the unresetting case, which is contained in \cite{C} and will be improved below (Theorem \ref{pi3 reflecting}):


\begin{thm}{\label{sigma2 admissible wITRM strength}}
	If $\alpha$ is $\Sigma_{2}$-admissible, then COMP$^{\text{wITRM}}_{\alpha}=\Delta_{1}(L_{\alpha})\cap\mathfrak{P}(\alpha)$.
\end{thm}

\section{$\alpha$-ITRMs and ZFC$^{-}$}

For the definition of $\alpha$-ITRMs, we refer to \cite{K1} or our sketch above. A set $x\subseteq\alpha$ is $\alpha$-ITRM-decidable or $\alpha$-ITRM-computable if and only if there are an $\alpha$-ITRM-program $P$ and a parameter $\gamma<\alpha$ such that, for all $\iota<\alpha$, we have $P(\iota,\gamma)\downarrow=1$ if and only if $\iota\in x$ and otherwise, $P(\iota,\gamma)\downarrow=0$. We denote the set of ITRM-computable sets by COMP$^{\text{ITRM}}_{\alpha}$.

In \cite{C}, we showed that, if $\kappa$ is an uncountable regular cardinal, then COMP$_{\kappa}^{\text{ITRM}}=\mathfrak{P}(\kappa)\cap L_{\kappa+1}$. Here, we will explore the matter further to reach a much stronger result.

We will occasionally write liminf$(c_{\iota}:\iota<\delta)$ or min$\{c_{\iota}:\iota<\delta\}$ where $\{c_{\iota}:\iota<\delta\}$ is a set of tuples $(c_{\iota}^{1},...,c_{\iota}^{k})$ of ordinals of fixed finite length $k$ to denote the tuple consisting of the component-wise inferior limits, minima etc., i.e.  $(\text{liminf}_{\iota<\delta}c^{1}_{\iota},...,\text{liminf}_{\iota<\delta}c^{k}_{\iota})$ and $(\text{min}_{\iota<\delta}c^{1}_{\iota},...,\text{min}_{\iota<\delta}c^{k}_{\iota})$, respectively.

When talking about $\alpha$-ITRM-computations in this paper, we always mean that parameters are allowed, even though we will not mention it. Throughout the paper, $p$ will denote Cantor's ordinal pairing function. If $(X,\in)$ is an $\in$-structure and $f:\alpha\rightarrow X$ is surjective, we will call $\{p(\iota,\xi):f(\iota)\in f(\xi)\}$ an $\alpha$-code for $X$. When $\alpha$ is clear from the context, the prefix $\alpha$ will occasionally be dropped.

For this section, let $\alpha$ be an exponentially closed ordinal. 

We will be working with ZF$^{-}$, which is ZFC without choice and the power set axiom; to be more precise, we use the formulation of ZF$^{-}$ established in \cite{GJH} as the most natural one. 

Our goal is to show the following result:

\begin{thm}{\label{zf- and itrm}}
	If $\alpha$ is exponentially closed, then $L_{\alpha}\models\text{ZF}^{-}$ if and only if COMP$^{\text{ITRM}}_{\alpha}=L_{\alpha+1}\cap\mathfrak{P}(\alpha)$.
\end{thm}

Let us say that $\alpha$ is a ZF$^{-}$-ordinal if and only if $L_{\alpha}\models\text{ZF}^{-}$. We will frequently and freely use the following folklore characterization of ZF$^{-}$-ordinals, the proof of which we recall for the convenience of the less set-theoretically minded reader:

\begin{lemma}{\label{zf- char}}
	$\alpha$ is a ZF$^{-}$-ordinal if and only if $\alpha$ is regular in $L_{\alpha+1}$. 
\end{lemma}
\begin{proof}
	Suppose that $\alpha$ is singular in $L_{\alpha+1}$; pick $\delta<\alpha$, $f:\delta\rightarrow\alpha$ such that $f\in L_{\alpha+1}$ and $f$ maps $\delta$ cofinally into $\alpha$. Towards a contradiction, assume that $L_{\alpha}\models$ZF$^{-}$. By assumption, $f$ is definable over $L_{\alpha}$; let $\phi$ be a formula, $\vec{p}\subseteq\alpha$ finite such that $f(\iota)=\xi$ holds if and only if $L_{\alpha}\models\phi(\iota,\xi,\vec{p})$. Then $f$ is a functional class in $L_{\alpha}$ and thus, by replacement in $L_{\alpha}$, $f[\delta]$ is an element of $L_{\alpha}$. But then, by the axiom of union in $L_{\alpha}$, we also have $\alpha=\bigcup f[\delta]\in L_{\alpha}$, a contradiction.
	
	Now assume that $\alpha$ is regular in $L_{\alpha+1}$. The only axioms to check are comprehension and collection (all other axioms of ZF$^{-}$ hold in all limit levels of the $L$-hierarchy that contain $\omega$). As comprehension is a consequence of collection, it suffices to deal with the latter. So let $\phi$ be a formula and $X,\vec{p}\in L_{\alpha}$ such that $L_{\alpha}\models\forall{x\in X}\exists{y}\phi(x,y,\vec{p})$. We need to show that there is $Y\in L_{\alpha}$ such that $L_{\alpha}\models\forall{x\in X}\exists{y\in Y}\phi(x,y,\vec{p})$. Let $\beta$ be minimal such that $X\in L_{\beta}$. Then $\beta<\alpha$ (as $\alpha$ is a limit ordinal) and moreover, by standard finestructure, there is a bijection $g:\beta\rightarrow L_{\beta}$ in $L_{\beta+1}$ and thus also a surjection $h:\beta\rightarrow X$. Now, the function $f:X\rightarrow\alpha$ mapping each $x\in X$ to the minimal $\gamma\in\alpha$ such that $L_{\gamma}$ contains some $y$ with $L_{\alpha}\models\phi(x,y,\vec{p})$ is clearly definable over $L_{\alpha}$ and thus contained in $L_{\alpha+1}$. By regularity of $\alpha$ in $L_{\alpha+1}$, the image of $f\circ h$ must be bounded in $\alpha$, say be $\eta$. Then $Y=L_{\eta}$ is as desired.
\end{proof}



We now show the direction from left to right. 

\begin{prop}{\label{min liminf}}
	Let $(\alpha_{\iota}:\iota<\delta)$ be an increasing sequence of ordinals, where $\delta$ is a limit ordinal, and let $\alpha:=\text{sup}_{\iota<\delta}\alpha_{\iota}$. Let $(\beta_{\iota}:\iota<\alpha)$ be another sequence of ordinals (by definition, $\alpha$ is a limit ordinal). Moreover, let $\mu_{\xi}=\text{min}\{\beta_{\iota}:\iota<\alpha_{\xi}\}$ for $\xi<\delta$. 
	
	Then $\text{liminf}_{\iota<\alpha}\beta_{\iota}=\text{liminf}_{\xi<\delta}\mu_{\xi}$.
\end{prop}
\begin{proof}
	We first show that $\text{liminf}_{\iota<\alpha}\beta_{\iota}\leq\text{liminf}_{\xi<\delta}\mu_{\xi}$.
	
	Suppose for a contradiction that $\text{liminf}_{\xi<\delta}\mu_{\xi}<\text{liminf}_{\iota<\alpha}\beta_{\iota}$. Thus, for some $\gamma<\alpha$, we have $\beta_{\iota}>\text{liminf}_{\xi<\delta}\mu_{\xi}$ for all $\iota>\gamma$. But then, the same holds for $\mu_{\xi}$ as soon as $\alpha_{\xi}>\gamma$, contradicting the definition of $\text{liminf}_{\xi<\delta}\mu_{\xi}$.
	
	\bigskip
	
	Now we show that $\text{liminf}_{\iota<\alpha}\beta_{\iota}\geq\text{liminf}_{\xi<\delta}\mu_{\xi}$.
	
	Suppose for a contradiction that $\text{liminf}_{\iota<\alpha}<\text{liminf}_{\xi<\delta}\mu_{\xi}$.
	
	Thus, there is $\gamma<\delta$ such that, for all $\xi>\gamma$, we have $\mu_{\xi}>\text{liminf}_{\iota<\alpha}\beta_{\iota}$. But then, we have 
	$\beta_{\xi}>\text{liminf}_{\iota<\alpha}\beta_{\iota}$ for all sufficiently large $\xi$, contradicting the definition of $\text{liminf}_{\iota<\alpha}\beta_{\iota}$.
\end{proof}

The next lemma is a strenghthening of a lemma that was proved in joint discussion with Philipp Schlicht and originally published in \cite{C}, namely that the definability of $F_{P}^{\alpha^{n}}$ over $L_{\alpha}$ holds when $\alpha$ is an uncountable regular cardinal.

\begin{lemma}{\label{f g h defbl}}
	Let $L_{\alpha}\models\text{ZF}^{-}$, let $P$ be an $\alpha$-ITRM-program using $n$ many registers, and let $\beta$ be an ordinal. Denote by $F_{P}^{\beta}$ the function that maps a $P$-configuration $c=(l,r_{1},...,r_{n})$ to the $P$-configuration arising by running $P$ for $\beta$ many steps with initial configuration $c$.
	
	Moreover, let $G_{P}^{\beta}$ be the function that maps a $P$-configuration $c=(l,r_{1},...,r_{n})$ to the tuple $(l^{\prime},r_{1}^{\prime},...,r_{n}^{\prime})$, where $l^{\prime}$ is the minimal program line index in any $P$-configuration occuring when one runs $P$ on $c$ for $\beta$ many steps; and similarly, $r_{1}^{\prime},...,r_{n}^{\prime}$ are the minimal contents of the registers $R_{1},...,R_{n}$ used by $P$ during this computation.
	
	Finally, for $0<n\in\omega$, $\hat{c}:=(\hat{l},\hat{r}_{1},...,\hat{r}_{n})\in\omega\times\alpha^{n}$, 
	let $H_{P}^{\beta,\hat{c}}$ be the function that maps a $P$-configuration $c$ to the binary sequence $(s_{0},...,s_{n})\in\{0,1\}^{n+1}$, where $s_{i}=1$ if and only if, for some configuration $c^{\prime}$ occuring during the $P$-computation of length $\beta$ starting with $c$, the $i$th component of $c^{\prime}$ coincides with the $i$th component of $\hat{c}$, and otherwise, $s_{i}=0$.
	
	Then $F_{P}^{\alpha^{k}}$, $G_{P}^{\alpha^{k}}$ and $H_{P}^{\alpha^{k}}$ are definable over $L_{\alpha}$, and in fact by $\Sigma_{4k}$-formulas for $0<k\in\omega$ (and thus in particular contained in $L_{\alpha+1}$).
\end{lemma}

\begin{proof}
	We will prove the definability of $F_{P}^{\alpha^{k}}$ and $G_{P}^{\alpha^{k}}$ by a simultaneous induction on $k$; the definability of $H_{P}^{\alpha^{k}}$ will then be an easy consequence of this proof.
	
	Let $k=1$, and let $c=(l,r_{1},...,r_{n})$. 
	
	As $\alpha$ is a limit ordinal, any partial computation of $P$ of length $\gamma<\alpha$ for any starting configuration will be contained in $L_{\alpha}$.
	
	Then, simply by expressing the liminf-rule, $F_{P}^{\alpha}(c)$ is definable over $L_{\alpha}$ as follows:
	
	$F_{P}^{\alpha}(c)=(l^{\prime},r_{1}^{\prime},...,r_{n}^{\prime})$ if and only if all of the following hold: 
	
	\begin{enumerate}
		
		\item For all $\iota<\alpha$, there is a $P$-computation of length $\iota+1$ starting with $c$ with active program line $l^{\prime}$ in the final configuration. This is expressable by a $\Pi_{2}$-formula.
		
		\item There is $\gamma<\alpha$ such that, for all $\iota\in(\gamma,\alpha)$, and all $P$-computations of length $\iota+1$ starting with $c$, the final configuration will have an active program line with index $\geq l^{\prime}$. This condition is expressable by a $\Sigma_{2}$-formula.
		
		\item For all $\rho_{i}<r_{i}^{\prime}$ ($1\leq i\leq n$), there is $\gamma_{i}<\alpha$ such that, for all $\iota>\gamma_{i}$, all $P$-computations of length $\iota+1$ starting with $c$ have an ordinal $\geq\rho_{i}$ in their $i$th register in their final configuration. This condition is expressable by a $\Pi_{3}$-formula.
		
		\item For all $\gamma<\alpha$, there is $\iota\in(\gamma,\alpha)$ such that there is a $P$-computation of length $\iota+1$ starting with $c$ such that, at time $\iota+1$, the $i$th register contains an ordinal $<r_{i}^{\prime}$ for all $1\leq i\leq n$. This condition is expressable by a $\Pi_{2}$-formula.
		
	\end{enumerate}
	
	Thus, $F_{P}^{\alpha}$ is $\Pi_{3}$-definable over $L_{\alpha}$.
	
	Moreover, $G_{P}^{\alpha}(c)=(\bar{l},\bar{r}_{1},...,\bar{r}_{n})$ holds if and only if the following conditions hold:
	
	\begin{enumerate}
		\item There are $\iota_{0},...,\iota_{n}<\alpha$ and $P$-computations 
		$c_{\iota_{0}},...,c_{\iota_{n}}$ starting with $c$ of length $\iota_{0}+1,...\iota_{n}+1$ respectively, such that $c_{\iota_{0}}$ has active program line $\bar{l}$, $c_{\iota_{1}}$ has content $\bar{r}_{1}$ in its first register, ..., $c_{\iota_{n}}$ has content $\bar{r}_{n}$ in its $n$th register in its final configuration. This is expressable by a $\Sigma_{1}$-formula.
		
		\item For all $\iota<\alpha$ and all $P$-computations of length $\iota+1$ starting with $c$, in the final configuration we have that the active program line index is $\geq\bar{l}$ and the content of the $i$th register is $\geq\bar{r}_{i}$, for all $1\leq i\leq n$. This is expressable by a $\Pi_{1}$-formula.
	\end{enumerate}
	
	Thus, $G_{P}^{\alpha}(c)$ is definable over $L_{\alpha}$ by the conjunction of a $\Sigma_{1}$-formula and a $\Pi_{1}$-formula.
	
	\bigskip
	
	Now assume that $F_{P}^{\alpha^{k}}$ and $G_{P}^{\alpha^{k}}$ have $\Sigma_{4k}$-definitions over $L_{\alpha}$. We show that $F_{P}^{\alpha^{k+1}}$ and $G_{P}^{\alpha^{k+1}}$ have $\Sigma_{4(k+1)}$-definitions over $L_{\alpha}$.
	
	For a arbitrary $P$-configuration $c$ and $\iota<\gamma<\alpha$, let us define $C_{\gamma}^{k}(\iota,c)$ (the sequence of every $\alpha^{k}$th configuration in the $P$-computation starting with $c$ up to time $\alpha^{k}\gamma$) and $D_{\gamma}^{k}(\iota,c)$ (the sequence of component-wise minima of configurations occuring in the $P$-computation starting with $c$ between times $\alpha^{k}\iota$ and $\alpha^{k}(\iota+1)$ up to time $\alpha^{k}\gamma$) by a simultaneous recursion as follows: 
	
	\begin{itemize}
		\item $C_{\gamma}^{k}(0,c)=c$
		\item $C_{\gamma}^{k}(\iota+1,c)=F_{P}^{\alpha^{n}}(C^{k}_{\gamma}(\iota,c))$
		\item $C_{\gamma}^{k}(\delta,c)=\text{liminf}_{\iota<\delta}D_{\gamma}^{k}(\iota,c)$ for $\delta<\gamma$ a limit ordinal.
		\item $D^{k}_{\gamma}(0,c)=c$
		\item $D^{k}_{\gamma}(\iota+1,c)=G_{\gamma}^{n}(C_{\gamma}^{k}(\iota,c))$
		\item $D^{k}_{\gamma}(\delta,c)=\text{min}\{D^{k}_{\gamma}(\iota,c):\iota<\gamma\}$ for $\delta<\gamma$ a limit ordinal.
	\end{itemize}
	
	By recursion (and the inductive assumption about the definability of $F_{P}^{\alpha^{k}}$ and $G_{P}^{\alpha^{k}}$ over $L_{\alpha}$) in $L_{\alpha}$, we have $(C_{\gamma}^{k}(\iota,c):\iota<\gamma)\in L_{\alpha}$ and $(D_{\gamma}^{k}(\iota,c):\iota<\gamma)\in L_{\alpha}$ for all $\gamma<\alpha$.
	
	Now, we can define $F_{P}^{\alpha^{k+1}}$ and $G_{P}^{\alpha^{k+1}}$ over $L_{\alpha}$ as follows:
	
	For $c=(l,r_{1},...,r_{n})$, we have 
	$F_{P}^{\alpha^{k+1}}(c)=\text{liminf}_{\iota<\alpha}D^{k}_{\iota+1}(\iota,c)$ by Proposition \ref{min liminf}. It is not hard to see, using the inductive assumption, that this is $\Sigma_{4k+4}$ over $L_{\alpha}$, as desired.
	
	On the other hand, we have that $G_{P}^{\alpha^{k+1}}(c)=\text{min}\{D_{\iota+1}^{k}(\iota,c):\iota<\alpha\}$: Clearly, the minimal value that the $i$th component assumes until time $\alpha^{k+1}$ is equal to the minimum of the minimal $i$th components occuring in each subinterval of the form $[\alpha^{k}\iota,\alpha^{k}(\iota+1))$.)
	
	In total, $F_{P}^{\alpha^{k+1}}$ and $G_{P}^{\alpha^{k+1}}$ are $\Sigma_{4(k+1)}$ over $L_{\alpha}$, as desired.
	
	\bigskip
	
	Finally, we turn to $H_{P}^{\alpha^{k},\hat{c}}$. 
	
	
	For $n=1$, we have $H_{P}^{\alpha,\hat{c}}=(s_{0},...,s_{n})$ if and only if, for any $i\in\{0,1,...,n\}$ such that $s_{i}=1$, there is $\gamma<\alpha$ such that, 
	after running $P$ on $c$ for $\gamma$ many steps, the $i$th component of the last configuration is equal to the $i$th component of $\hat{c}$ and for any $i\in\{0,1,...,n\}$ with $s_{i}=0$, this is false.
	
	The former condition is $\Sigma_{1}$ over $L_{\alpha}$, the latter is $\Pi_{1}$, so the whole definition is $\Sigma_{2}$.
	
	Now suppose that $H_{P}^{\alpha^{k},\hat{c}}$ is defined. By recursion in $L_{\alpha}$, define, for all $\gamma<\alpha$ and all $\iota<\gamma$:
	
	\begin{itemize}
		\item $E_{\gamma}^{k,\hat{c}}(0,c)=(0,...,0)$
		\item $E_{\gamma}^{k,\hat{c}}(\iota+1,c)=H_{P}^{\alpha^{k},\hat{c}}(C^{k}_{\gamma}(\iota,c))$
		\item $E_{\gamma}^{k,\hat{c}}(\delta,c)=(0,...0)$ for $\delta<\gamma$ a limit ordinal.
	\end{itemize}
	
	The let $H_{P}^{\alpha^{k+1},\hat{c}}(c)=\text{max}\{E_{\gamma+1}^{k,\hat{c}}(\gamma,c):\gamma<\alpha\}$ (where the maximum is also to be taken in each component separately).
	
	Clearly, this is definable over $L_{\alpha}$, as desired.
	
\end{proof}

\begin{lemma}{\label{early computables}}
	Suppose that $L_{\alpha}\models\text{ZF}^{-}$. 
	If $x\subseteq\alpha$ is computable by an $\alpha$-ITRM in time $<\alpha^{n}$ for some $n\in\omega$, then $x\in L_{\alpha+1}$.
\end{lemma}
\begin{proof}
	Suppose that $x\subseteq\alpha$ is decidable by an $\alpha$-ITRM-program $P$ and that, for each $\iota<\alpha$, $P(\iota)$ halts in $<\alpha^{n}$ many steps, where $n\in\omega$. 
	
	Now, for all $\iota\in\alpha$, we have $\iota\in x$ if and only if $P(\iota)\downarrow=1$ in $<\alpha^{n}$ many steps if and only if, at time $\alpha^{n}$, the first register contains $1$ in the $P$-computation starting in configuration $(1,\iota,0,...,0)$\footnote{After a halting configuration has been reached, the computation continues by repeating this configuration without changes.} if and only if $F_{P}^{\alpha^{n}}(1,\iota,0,...,0)$ has $1$ in its second component. By Lemma \ref{f g h defbl}, the last condition is $\Sigma_{4n}$ over $L_{\alpha}$. 
	
	In particular, it follows that $x\in L_{\alpha+1}$.
\end{proof}

The proof actually shows more:

\begin{corollary}
	If $x\subseteq\alpha$ is $\alpha$-ITRM-decidable with time bound $\alpha^{k}$, then $x$ is $\Sigma_{4k}$ over $L_{\alpha}$. 
	
	Moreover, if $\alpha$ is $\Sigma_{4k}$-admissible, then any $x\subseteq\alpha$ that is computable by an $\alpha$-ITRM with time bound $\alpha^{k}$ is an element of $L_{\alpha+1}$.
	
	Without any assumption on $\alpha$, we still obtain that, if $x\subseteq\alpha$ is $\alpha$-ITRM-decidable with time bound $\alpha^{k}$, then $x\in L_{\alpha+k+1}$.
\end{corollary}
\begin{proof}
	The first two claims are clear. For the third, note that, inductively, $C^{k}_{\gamma}$ and $D_{k}^{\gamma}$ will be contained in $L_{\alpha+k}$, so that $F_{P}^{\alpha^{k+1}}$ and $G_{P}^{\alpha^{k+1}}$ are definable over $L_{\alpha+k}$ and thus contained in $L_{\alpha+k+1}$, so that now any $x\subseteq\alpha$ that is $\alpha$-ITRM-computable with time bound $\alpha^{k+1}$ is contained in $L_{\alpha+k+2}$. Thus, an induction on $k$ proves the desired result.
\end{proof}

\begin{lemma}{\label{koepke seyfferth}} [Koepke and Seyffert, see \cite{KS}]
	Let $\alpha$ be exponentially closed. Then $x\subseteq\alpha$ is $\alpha$-ITRM-computable in $\alpha$ many steps if and only if $x$ is $\Delta_{1}$ over $L_{\alpha}$.
	
	In particular, there is an $\alpha$-ITRM-computable code $c\subseteq\alpha$ for $L_{\alpha}$ in which each ordinal $\iota$ is coded by $\iota+1$.
	
	Moreover, the truth predicate for bounded formulas in $L_{\alpha}$ with parameters is $\alpha$-ITRM-decidable.
\end{lemma}
\begin{proof}
	For the first and last claim, see Koepke and Seyfferth \cite{KS}.\footnote{Strictly speaking, \cite{KS} proves this for $\alpha$-ITTMs, but the adaptation to $\alpha$-ITRMs is straightforward, see \cite{C}, Theorem 3.3.3.} For the second claim, it is easy to see that such a code is $\Delta_{1}$ over $L_{\alpha}$.
\end{proof}

\begin{lemma}{\label{succ level comp}}
	Let $\alpha$ be exponentially closed, $x\in L_{\alpha+1}\cap\mathfrak{P}(\alpha)$. Then $x$ is $\alpha$-ITRM-computable.
\end{lemma}
\begin{proof}
	By Lemma \ref{koepke seyfferth}, a subset $c\subseteq\alpha$ coding $L_{\alpha}$ is $\alpha$-ITRM-computable as it is easy to see that there is such a code which is $\Delta_1$ over $L_{\alpha}$: More precisely, associate with every $\delta\alpha$ a triple $(\beta,k,\gamma)\in\alpha\times\omega\times\alpha$, which is meant to represent $\text{Def}(\beta,k,\gamma):=\{x\in L_{\beta}:L_{\beta}\models\phi_{k}(x,\gamma)\}$. (In particular then, $\iota<\alpha$ is represented by $(\iota,k,\emptyset)$, where $k$ is an index for the formula ``$x$ is an ordinal''. Now define a code $c\subseteq\alpha$ by saying that $p(\iota,\xi)\in c$ if and only if there are $t_{0}:=(\beta_{0},k_{0},\gamma_{0})$, 
	$t_{1}:=(\beta_{1},k_{1},\gamma_{1})$ in $\alpha\times\omega\times\alpha$ such that $\iota$ codes $t_{0}$, $\xi$ codes $t_{1}$ and $\text{Def}(t_{0})\in\text{Def}(t_{1})$ if and only if we have $\text{Def}(t_{0})\in\text{Def}(t_{1})$ for all such $t_{0}$, $t_{1}$. This definition is clearly $\Delta_{1}$ over $L_{\alpha}$. 
	
	Now suppose that $x\in L_{\alpha+1}$ is given as 
	$$x=\{\iota<\alpha:L_{\alpha}\models\phi(\iota,\vec{p})\}\text{,}$$
	where $\vec{p}$ is a finite sequence in $\alpha$. We show by induction on the complexity of $\phi$ that $x$ is $\alpha$-ITRM-decidable.
	
	Suppose that $\phi$ is written in the form $\exists{x_{1}}\forall{x_{2}}...\forall{x_{n}}\psi$, where $\psi$ is quantifier-free.
	
	To evaluate $\psi$, one only needs to use the algorithm for evaluating the bounded truth predicate from Lemma \ref{succ level comp}.
	
	Then, for each quantifier alternation, we perform an exhaustive search through $\alpha$. For $n$ quantifier alternations, $n$ extra registers are used for the nested searches.
	
	More specifically, if $Q$ decides $\{(\iota,\xi)\in\alpha\times\alpha:\phi(\iota,\xi)\}$ for some formula $\phi$, then $\{\iota<\alpha:\exists{\xi}\phi(\iota,\xi)\}$ is decided by running through $\alpha$ in a new register and using $Q$ on each content of that register to decide whether $\phi(\iota,\xi)$ holds. If this terminates (i.e., if the new register contains $0$, due to an overflow) without $Q$ ever having returned the output $1$, then $\exists{\xi}\phi(\iota,\xi)$ is false, otherwise, it is true.
	
	The set $\{\iota<\alpha:\forall{\xi}\phi(\iota\xi)\}$ is just the relative complement of the 
	set $\{\iota<\alpha\exists{\xi}\neg\phi(\iota,\xi)\}$ in $\alpha$ and can thus be decided similarly.
\end{proof}

We now work towards a bound on the halting times on $\alpha$-ITRM-programs when $\alpha$ is a ZF$^{-}$-ordinal. Our approach is an adaptation of the proof by Koepke in \cite{K1} that the halting times of $\omega$-ITRMs are bounded by $\omega_{\omega}^{\text{CK}}$ and strengthens our result from \cite{C} that the halting times of $\kappa$-ITRMs are bounded by $\kappa^{\omega}$ when $\kappa$ is an uncountable regular cardinal.

The following lemma generalizes the looping criterion for ITRMs from \cite{KM}.

\begin{lemma}{\label{looping criterion}}
	Let $P$ be an $\alpha$-ITRM-program. Suppose that, during the computation of $P$, there are times $\iota<\xi$ such that the configurations at time $\iota$ and $\xi$ are equal and such that any configuration arising in between is in every component $\geq$ the configuration at time $\iota$. Then $P$ is looping, repeating its behaviour between times $\iota$ and $\xi$ and in particular never halts.
\end{lemma}
\begin{proof}
	Let $\delta$ be such that $\iota+\delta=\xi$.
	By the liminf-rule, the configuration at time $\iota$ reappears at any time of the form $\xi+\delta\gamma$.
\end{proof}

\begin{defini}
	In the situation of Lemma \ref{looping criterion}, we say that $(\iota,\xi)$ witnesses the looping of $P$.
\end{defini}

A new phenomenon occuring for $\alpha$-ITRMs with $\alpha>\omega$, but not for ITRMs is the possibility that a register content $>0$ occurs at a limit time without ever having been contained in that register before; for example, if one counts upwards in a register, starting with $0$, then at time $\omega$, this register will contain $\omega$ for the first time. This kind of limits complicates the control over the register contents that we need to ensure looping. Fortunately, for reasonable closed $\beta$, we can show that it cannot occur at time $\beta$.

\begin{defini}
	Let $P$ be an $\alpha$-(w)ITRM-program, $\iota,\delta<\alpha$, $\delta>0$ and $\tau$ a limit ordinal. We say that $\delta$ is a proper limit of $(P,\iota)$ at time $\tau$ if and only if some register contains $\delta$ at time $\tau$ in the computation of $P$ in the input $\iota$, but that register had contents $<\delta$ cofinally often before time $\tau$.
\end{defini}

\begin{lemma}{\label{no bad limits}}
	Let $\alpha$ be a ZF$^{-}$-ordinal, let $P$ be an $\alpha$-ITRM-program, $k\in\omega$, and let $R$ be a register used by $P$ and let $r>0$ be its content at time $\alpha^{k}$. Then there is $\tau<\alpha^{k}$ such that all contents of $R$ after time $\tau$ were $\geq r$ and moreover, $r$ was cofinally often the content of $R$ before time $\alpha^{k}$.
\end{lemma}
\begin{proof}
	By the liminf rule, the second claim follows from the first. It thus suffices to show the first claim.
	
	Suppose for a contradiction that the first claim fails. Note that, as a register content of an $\alpha$-ITRM, we have $r<\alpha$. Now we have that, for any $\rho<r$, there is a minimal ordinal $\iota(\rho)<\alpha$ such that, from time $\alpha^{k-1}\iota(\rho)$ on, all contents of $R$ were $\geq\rho$ (but cofinally often, it was $<r$). Consider the function $\rho\mapsto\iota(\rho)$, which maps $r<\alpha$ cofinally into $\alpha$.
	
	We claim that this function is definable over $L_{\alpha}$, hence contained in $L_{\alpha+1}$, contradicting the assumption that $\alpha$ is regular in $L_{\alpha+1}$. To this end, we recall from the proof of Lemma \ref*{f g h defbl} that $(D_{\iota}^{k}:\iota<\gamma)$ is definable over $L_{\alpha}$ for every $\gamma<\alpha$.
	
	Now $\beta\geq\iota(\rho)$ holds if and only if, for all $\xi\in[\beta,\alpha)$, we have that $D^{k}_{\xi+1}(\beta,c)$ has an ordinal $\geq\rho$ in its $(i+1)$st component.
	
	Clearly, this is expressable by some $\in$-formula over $L_{\alpha}$ since $D^{k}_{\xi+1}(\beta,c)$ is so expressable uniformly in $\zeta$, $\beta$ and $c$ (where $c$ is the initial configuration). Hence, the minimal such ordinal is also definable over $L_{\alpha}$ and thus, so is the function $\iota\mapsto\iota(\rho)$.
\end{proof}

\begin{thm}{\label{halting time bound}}
	Let $\alpha$ be a ZF$^{-}$-ordinal. Then an $\alpha$-ITRM-program using $n$ registers halts in $<\alpha^{n+1}$ many steps or not at all.
\end{thm}
\begin{proof}
	We follow the argument by Koepke from \cite{K1}.
	
	We actually show that, if an $\alpha$-ITRM-computation with $L_{\alpha}\models$ZF$^{-}$ reaches time $\alpha^{n+1}$, then at least $n$ registers must contain $0$ or there are $\iota,\xi<\alpha^{n+1}$ witnessing the looping of $P$. Since there cannot be more registers containing $0$ than there are registers in total, this proves the claim.
	
	By Lemma \ref{no bad limits}, we know that, if an $\alpha$-ITRM-computations reaches time $\alpha^{k}$ for some $1\leq k\in\omega$, then there is $\tau<\alpha^{k}$ (namely, the maximum of the values guaranteed to exist by that lemma for each of the finitely many registers) such that, from time $\tau$ on up to time $\alpha^{k}$, no register content dropped below its value at time $\alpha^{k}$ (this holds trivially when that content is $0$). By increasing $\tau$ if necessary, we can assume without loss of generality that the same holds for the active program line.  Let $c$ be the configuration of $P$ at time $\alpha^{k}$.
	
	Now let $n=1$ and pick $\tau$ as just described. We build an increasing sequence of ordinals $(\alpha_{k}:k\in\omega)$ such that $\alpha_{0}=\tau$ and, for all $k\in\omega$, $\alpha_{k}>\text{max}\{\alpha_{0},...,\alpha_{k-1}\}$ is minimal such that the $(k$ mod $(n+1))$th component of the configuration at time $\alpha_{k}$ agrees with that of $c$. (Such a sequence exists since, by assumption, none of the register contents at time $\alpha$ is due to an overflow.) 	
	Clearly, this sequence is definable over $L_{\alpha}$ as a map from $\omega$ into $\alpha$ and is thus not cofinal. 
	
	Let $\eta=\text{sup}\{\alpha_{k}:k\in\omega\}$. Then $\tau<\eta<\alpha$ and the $P$-configurations at time $\eta$ is equal to $c$ by the liminf-rule.
	
	But then, $(\eta,\alpha)$ witnesses the looping of $P$ and thus, $P$ does not halt.\footnote{This argument actually shows that the halting times of $\alpha$-wITRMs are bounded by $\alpha$ when $\alpha$ is $\Sigma_{2}$-admissible. See \cite{C} and the remark in the introduction.}

	\bigskip

	Let us now assume that the theorem holds for $n$. Suppose the computation arrives at time $\alpha^{n+1}$ and less than $n$ many registers contain $0$ at that time. Again, pick $\tau$ as in the first case.
	
	\bigskip
	
	Suppose first that there is no register overflow at time $\alpha^{n+1}$.
	
	Once more, we want to build the sequence $(\alpha_{k}:k\in\omega)$ as for $n=1$. However, as it stands, this would be a sequence of ordinals $<\alpha^{n+1}$ and it would be defined over $L_{\alpha^{n+1}}$, not over $L_{\alpha}$, which would not help much to see that it is bounded.
	
	We thus modify the definition a bit: $\alpha_{0}$ will be the minimal $\zeta<\alpha$ such that $\alpha^{n}\zeta>\tau$. After that, $\alpha_{k+1}$ will be the minimal ordinal such that $\alpha_{k+1}>\text{max}\{\alpha_{0},...,\alpha_{k}\}$ and the $((k+1)$ mod $(n+1))$st component will at some time between $\alpha^{n}\alpha_{k+1}$ and $\alpha^{n}(\alpha_{k+1}+1)$ agree with the corresponding component of $c$.
	
	This is again a sequence of elements of $\alpha$ and, by Lemma \ref{f g h defbl}, it is definable over $L_{\alpha}$ and thus bounded in $\alpha$. 
	
	With $\eta=\text{sup}\{\alpha_{k}:k\in\omega\}$, we thus have $\eta<\alpha$, hence $\alpha^{n}\eta<\alpha^{n+1}$ and, by the liminf-rule, the configuration at time $\alpha^{n}\eta$ will be $c$. Consequently, the looping criterion for $P$ is once again satisfied, and $P$ does not halt.
	
	Thus, if there is no register overflow at time $\alpha^{n+1}$, then $P$ does not halt.
	
	\bigskip

	Now suppose that there is a register overflow at time $\alpha^{n+1}$. Pick one register of $P$, say $R$, that overflows at time $\alpha^{n+1}$. Thus, we can now chose the $\tau$ above additionally in such a way that, after time $\tau$ and up to time $\alpha^{n+1}$, $R$ never contains $0$.
	
	Consider the configuration $\bar{c}$ at time $\tau+\alpha^{n}$, which we can regard as arising by running $P$ for $\alpha^{n}$ many steps on the configuration it had at time $\tau$.
	
	In $\bar{c}$, no register contains $0$ that does not contain $0$ at time $\alpha^{n+1}$ by assumption on $\tau$, and by the same reason, $R$ does not contains $0$. Thus, at time $\tau+\alpha^{n}$, at most $(n-1)$ registers contain $0$. As we can regard this as step $\alpha^{n}$ in a $P$-computation starting in the configuration at time $\tau$, it follows again that $P$ is looping.
	
	\bigskip

	Both cases are finished. Thus, if $P$ does not have at least $n$ many $0$s in its registers by time $\alpha^{n+1}$, it does not halt.
	
	In particular, this holds if $P$ uses $<n$ many registers.
\end{proof}

By exploiting the proof a bit further, we obtain the following refinement:

\begin{corollary}{\label{refined halting bound}}
	For any $k\in\omega$, there is $n(k)\in\omega$ with the following property: If $L_{\alpha}\models\Sigma_{n(k)}$-collection, then an $\alpha$-ITRM-program using $\leq k$ many registers halts or loops in $<\alpha^{k+1}$ many steps. In fact, there is a natural constant $C$ such that $n(k)\leq C\cdot k$ for all $k\in\omega$.
\end{corollary}

\begin{corollary}{\label{itrm uniform bound}}
	An $\alpha$-ITRM-program $P$ with $L_{\alpha}\models$ZF$^{-}$ halts in $<\alpha^{\omega}$ many steps or does not halt at all.
\end{corollary}
\begin{proof}
	$P$ uses some natural number $n$ of registers. Thus, if it halts, it halts before time $\alpha^{n+1}<\alpha^{\omega}$.
\end{proof}

\begin{corollary}{\label{clockable supremum}}
	If $L_{\alpha}\models$ZF$^{-}$, then $\alpha^{\omega}$ is the supremum of the $\alpha$-ITRM-halting times.
\end{corollary}
\begin{proof}
	That $\alpha^{\omega}$ is an upper bound was just proved. 
	On the other hand, it is easy to see that, for any $n\in\omega$, there is an $\alpha$-ITRM-program that halts at time $\alpha^{n}$:
	
	To halt at time $\alpha$, just count upwards in some register, starting with $1$ and halt once that register overflows (i.e., contains $0$).
	
	Now, if $P$ halts at time $\alpha^{n}$, take a program as above, but before incrementing its register, run $P$ once each time. Clearly, this halts at time $\alpha^{n+1}$.
\end{proof}

We are ready to prove the first direction of our main result.

\begin{thm}{\label{left to right}}
	Suppose that $\alpha$ is a ZF$^{-}$-ordinal. Then COMP$^{\text{ITRM}}_{\alpha}=L_{\alpha+1}\cap\mathfrak{P}(\alpha)$.
\end{thm}
\begin{proof}
	Let $x$ be $\alpha$-ITRM-computable by the $\alpha$-ITRM-program $P$. Suppose that $P$ uses $n$ registers. Then $P$ runs for $<\alpha^{n+1}$ many steps on each input by Theorem \ref{halting time bound}. Hence $x$ is $\alpha$-ITRM-decidable with time bound $\alpha^{n+1}$. By Lemma \ref{early computables}, it follows that $x$ is definable over $L_{\alpha}$. Hence $x\in L_{\alpha+1}$.
	
	On the other hand, if $x\in L_{\alpha+1}$, then, by Lemma \ref{succ level comp}, $x$ is $\alpha$-ITRM-decidable.
\end{proof}

\begin{remark}
	As a consequence of the last theorem, it follows that, for $L_{\alpha}\models$ZF$^{-}$, $\alpha$-ITRMs cannot evaluate truth predicates for $L_{\alpha}$ (since such a truth predicate would allow us to compute sets outside of $L_{\alpha+1}$, see below.) In fact, together with the results below, this is possible if and only if $L_{\alpha}\not\models$ZF$^{-}$.
\end{remark}

\begin{remark}
We point out that, for $\alpha$ a ZF$^{-}$-ordinal, the realm of computable objects for $\alpha$-ITRMs is $L_{\alpha+1}$, which is only a very minor portion of $L_{\alpha^{\omega}}$, the first $L$-level containing all halting $\alpha$-ITRM-computations. To our knowledge, this is the first time such a divergence between levels containing computations and levels containing the computable objects has occured in ordinal computability. (See, however, footnote $4$ below.)
\end{remark}

\bigskip

We now work towards the reverse direction. To this end, we introduce some terminology from \cite{C}.

\begin{defini}{\label{singular}}
	An ordinal $\alpha$ is (w)ITRM-singular if and only if there are $\beta<\alpha$, a cofinal function $f:\beta\rightarrow\alpha$ and an $\alpha$-(w)ITRM-program $P$, $\xi<\alpha$ such that $P(\iota,\xi)\downarrow=f(\iota)$ for all $\iota\in\beta$. 
\end{defini}

We note that the singularising functions can be chosen to have a particularly nice form:

\begin{prop}{\label{continuous itrm-singularization}}
	If $\alpha$ is (w)ITRM-singular, then there is a function $g:\gamma\rightarrow\alpha$ with $\gamma<\alpha$ such that $f[\gamma]$ is unbounded in $\alpha$ and such that $g$ is $\alpha$-(w)ITRM-computable, continuous and increasing.
\end{prop}
\begin{proof}
	Let $f:\beta\rightarrow\alpha$ with $\beta<\alpha$ be such that $f[\beta]$ is unbounded in $\alpha$ and $f$ is $\alpha$-(w)ITRM-computable. 
	Now pick $\rho\leq\beta$ minimal such that $f[\rho]$ is unbounded in $\alpha$ and define $g:\rho\rightarrow\alpha$ by 
	$g(\xi)=\text{sup}\{f(\iota):\iota<\xi\}$ for $\xi<\rho$. This can be computed on an $\alpha$-(w)ITRM as follows: Given $\xi<\rho$ in the input register, successively compute the values of $f(\iota)$ for all $\iota<\xi$. At the start of this computation, store $f(0)$ in some extra register $R$. Whenever some $f(\iota)$ is larger than the current content of that register, replace the content of that register with $f(\iota)$. When we reach $\xi$, that register will contain $g(\xi)$. Then $g$ is as desired. 
\end{proof}

\begin{lemma}{\label{itrm vs succ level}}
	Suppose that $\alpha$ is exponentially closed. Then $\alpha$ is ITRM-singular if and only if $\alpha$ is singular in $L_{\alpha+1}$, i.e. if and only if $L_{\alpha}\not\models\text{ZF}^{-}$.
\end{lemma}
\begin{proof}
	Suppose first that $\alpha$ is singular in $L_{\alpha+1}$. Then there is $f\in L_{\alpha+1}$ such that $f$ maps some $\delta<\alpha$ cofinally into $\alpha$. By a simple coding, we can regard $f$ as a subset of $\alpha$. By Lemma \ref{succ level comp}, $f$ is $\alpha$-ITRM-computable. Thus, $\alpha$ is ITRM-singular.
	
	Now suppose that $\alpha$ is ITRM-singular. Suppose for a contradiction that $L_{\alpha}\models$ZF$^{-}$. Let $f$ be an $\alpha$-ITRM-computable function that maps some $\delta<\alpha$ cofinally into $\alpha$. By coding, we can regard $f$ as a subset of $\alpha$. By Theorem \ref{left to right}, we have $f\in L_{\alpha+1}$. Thus $\alpha$ is singular in $L_{\alpha+1}$, a contradiction. 
\end{proof}

We will also use the following theorem from \cite{C} (Theorem 3.3.28):

\begin{lemma}{\label{singular truth predicate evaluation}}
	If $\alpha$ is ITRM-singular, then there is an $\alpha$-ITRM-program $P_{\text{truth}}$ such that, for all $\gamma<\alpha$ and all $n\in\omega$, we have $P(n,\gamma)\downarrow=1$ if and only if $L_{\alpha}\models\phi_{n}(\gamma)$ and otherwise, we have $P(n,\gamma)\downarrow=0$.
\end{lemma}

Although we do not give a detailed proof of this result here and rather refer to \cite{C}, we offer a sketch for the interested reader to see how ITRM-singularity enters the picture. Let a code $c$ for $L_{\alpha}$ be given, where the element coded by $\iota<\alpha$ in $c$ is given by $h(\iota)$. Moreover, let us say that the statement to be evaluated is $\exists{x_{1}}\forall{x_{2}}...\exists{x_{n-1}}\forall{x_{n}}\psi$, where $\psi$ is quantifier-free. Thus, $\psi$ is a Boolean combination of statements that can be read off from $c$ for all assignments, which can easily be done by an $\alpha$-ITRM with access to $c$.

 Below, we will represent a sequence $(\alpha_{0},...,\alpha_{k})$ of ordinals using two stacks, one of which contains $k$, while the other contains $p(\alpha_{0},...,\alpha_{k})$, which is defined thus: $p(\alpha_{0})=\alpha_{0}$, $p(\alpha_{0},\alpha_{1})$ is Cantors's ordinal pairing function, and $p(\alpha_{0},...,\alpha_{k+1})=p(p(\alpha_{0},...,\alpha_{k}),\alpha_{k+1})$.

Now, what we would like to do - and what is done e.g. in the work of Koepke on ORMs - is the following: Store $0$ on the bottom of a stack. We now want to test whether $\forall{x_{2}}...\exists{x_{n-1}}\forall{x_{n}}\psi$ holds when one substitutes $x_{1}$ with $h(0)$ in $\psi$. To this end, we successively put all elements of $\alpha$ on the stack, considering $p(0,\iota)$ for all $\iota<\alpha$ one after the other. For each such pair, we then want to evaluate whether $\exists{x_{3}}\forall{x_{4}}...\exists{x_{n-1}}\forall{x_{n}}\psi$ holds when one substitutes $h(0)$ for $x_{1}$ and $h(\iota)$ for $x_{2}$, which is now done by putting further elements on the stack. If the answer is ``yes'' for each $\iota$, we finally halt and return `true'. If the answer is ``no'' for some $\iota$, we replace $0$ with $1$ and repeat the whole procedure, and so on, until we have either found a witness for the $x_{1}$ or have run through the whole of $\alpha$ and thus know that none exists.

However, as it stands, this does not work (in fact, by our remark above on the inability of $\alpha$-ITRMs to evaluate truth predicates in $L_{\alpha}$ when $L_{\alpha}\models$ZF$^{-}$, it cannot). The reason is this: After, e.g., considering $p(0,i)$ for all $i\in\omega$, the stack register will not contain $p(0,\omega)$, as it should, but simply $\omega$, thus losing all information. The same happens frequently when the stack contents approach limits.

Fortunately, there is a way out: As is already observed and used in \cite{K1}, the following is true:

\begin{lemma}{\label{limit preservation}}
	Let $(\beta_{\iota}:\iota<\delta)$ be a sequence of ordinals of limit length, $\beta=\text{liminf}_{\iota<\delta}\beta_{\iota}$, and let $\alpha>\beta$. 
	Then $\text{liminf}_{\iota<\delta}p(\alpha,\beta_{\iota})=p(\alpha,\beta)$.
\end{lemma}

Thus, sequence coding is compatible with limits, provided the limits are small enough. And if $\alpha$ is ITRM-singular, this can be exploited: Suppose that $f:\delta\rightarrow\alpha$ with $\delta<\alpha$ is cofinal, total and $\alpha$-ITRM-computable. Note that, by putting $\delta^{2}$ at the bottom of the stack, we can perform a depth-first-search through $\delta^{<\omega}$ on an $\alpha$-ITRM. Take an auxiliar register $R^{*}$ in which this is done, called the ``regulating register''. For the sake of simplicity, we will consider in detail only the case that the formula is $\exists{x}\forall{y}\psi$; the rest is then a matter of iteration. We proceed as follows: at each time, $R^{*}$ contains a sequence $(\delta^{2},\iota_{0},\iota_{1})$ with $\iota_{0},\iota_{1}<\delta$. 
At the same time, the ``main register'' $R$ will contain a sequence 
$(f(\iota_{0}),\zeta_{0},f(\iota_{1}),\zeta_{1})$ with $\zeta_{i}<f(\iota_{i})$ for $i\in\{1,2\}$. For any such sequence, it is tested whether $\psi$ holds for the elements coded by $\zeta_{0}$ and $\zeta_{1}$ using the code $c$. If not, the current candidate for $\zeta_{0}$ is not good and we change the contents of $R$ first to $(f(\iota_{0})+1,\zeta_{0})$, then 
to $(f(\iota_{0})+1,\zeta_{0}+1)$ and then to $(f(\iota_{0})+1,\zeta_{0}+1,f(\iota_{1})+1,0)$, so that the search in the second component can continue. 
If yes, we simply replace $\zeta_{1}$ with $\zeta_{1}+1$. When $\zeta_{1}=f(\iota_{1})$, we have successfully checked up to $f(\iota_{1})$; in this case, we increase $\iota_{1}$ by $1$ and modify the contents of $R$ as follows: $(f(\iota_{0})+1,\zeta_{0},f(\iota_{1})+1,f(\iota_{1}))\mapsto(f(\iota_{0},\zeta_{0})\mapsto(f(\iota_{0})+1,\zeta_{0},f(\iota_{1}+1)+1,0)$ and continue. When $\zeta_{0}=f(\iota_{0})$, we have unsuccessfully searched for a witness below $f(\iota_{0})$; in that case, we modify the content of $R^{*}$ first to $(\delta^{2},\iota_{0}+1)$ and then to $\delta^{2},\iota_{0}+1,0)$ and moreover, we modify the current content $r$ of $R$ as follows: 
$r\mapsto(f(\iota_{0}+1)+1)\mapsto(f(\iota_{0}+1)+1,0)\mapsto(f(\iota_{0}+1)+1,0,f(0)+1,0)$. When $\iota_{1}=\delta$, we have successfully checked all ordinals $<\alpha$ and $\exists{x}\forall{y}\psi$ holds, where $x$ is the element coded by the current value of $\zeta_{0}$; thus, we output ``true''. On the other hand, when $\iota_{0}=\delta$, we have unsuccessfully searched through $\alpha$ for a candidate for $x$, in which case we output ``false''.


The following picture illustrates the approach.

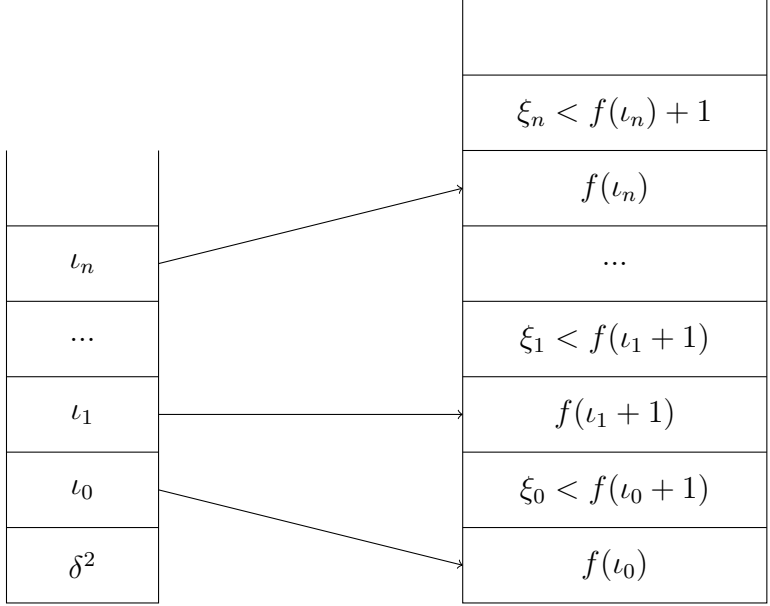
\begin{tikzpicture}
\draw (0,6)--(0,0)--(2,0)--(2,6);
\draw (0,1)--(2,1);
\node at (1,0.5) {$\delta^{2}$};
\draw (0,2)--(2,2);
\node at (1,1.5) {$\iota_{0}$};
\draw (0,3)--(2,3);
\node at (1,2.5) {$\iota_{1}$};
\draw (0,4)--(2,4);
\node at (1,3.5) {...};
\draw (0,5)--(2,5);
\node at (1,4.5) {$\iota_{n}$};

\draw (6,8)--(6,0)--(10,0)--(10,8);
\draw (6,1)--(10,1);
\node at (8,0.5) {$f(\iota_{0})$};
\draw (6,2)--(10,2);
\node at (8,1.5) {$\xi_{0}<f(\iota_{0}+1)$};
\draw (6,3)--(10,3);
\node at (8,2.5) {$f(\iota_{1}+1)$};
\draw (6,4)--(10,4);
\node at (8,3.5) {$\xi_{1}<f(\iota_{1}+1)$};
\draw (6,5)--(10,5);
\node at (8,4.5) {...};
\draw (6,6)--(10,6);
\node at (8,5.5) {$f(\iota_{n})$};
\draw (6,7)--(10,7);
\node at (8,6.5) {$\xi_{n}<f(\iota_{n})+1$};

\draw [->] (2,1.5)--(6,0.5);
\draw [->] (2,2.5)--(6,2.5);
\draw [->] (2,4.5)--(6,5.5);
\end{tikzpicture}

\bigskip

Given this lemma, the rest is a matter of a standard diagonalization:

\begin{thm}{\label{no zf- to more than succ}}
	Suppose that $\alpha$ is exponentially closed and $L_{\alpha}\not\models$ZF$^{-}$. Then 
	COMP$^{\text{ITRM}}_{\alpha}\neq L_{\alpha+1}\cap\mathfrak{P}(\alpha)$.
\end{thm}
\begin{proof}
	
	Suppose that $\alpha$ is not a ZF$^{-}$-ordinal. By Lemma \ref{zf- char}, it follows that $\alpha$ is singular in $L_{\alpha+1}$. By Lemma \ref{itrm vs succ level}, $\alpha$ is ITRM-singular. By Lemma \ref{singular truth predicate evaluation}, there is a program $P_{\text{truth}}$ that evaluates truth predicates in $L_{\alpha}$. 
	
	For $\gamma\in\text{On}$, let us write $n(\gamma)$ for the unique natural number $n$ and $\xi(\gamma)$ for the unique ordinal $\xi$ such that $\gamma$ can be written in the form $\gamma=\omega\xi+n$.
	
	Now let $D:=\{\iota<\alpha:L_{\alpha}\not\models\phi_{n(\iota)}(\iota,\xi(\iota))\}$. Using $P_{\text{truth}}$, $D$ is clearly $\alpha$-ITRM-decidable.
	
	Assume for a contradiction that $D\in L_{\alpha+1}$. Thus, there are $k\in\omega$ and $\xi<\alpha$ such that $D=\{\iota<\alpha:L_{\alpha}\models\phi_{k}(\iota,\xi)\}$. 
	Let $\gamma:=\omega\xi+k$. Then $L_{\alpha}\models\phi_{k}(\gamma,\xi)\Leftrightarrow\gamma\in D\Leftrightarrow L_{\alpha}\not\models\phi_{n(\gamma)}(\gamma,\xi(\gamma))\Leftrightarrow L_{\alpha}\not\models\phi_{k}(\gamma,\xi)$, a contradiction. Thus $D\notin L_{\alpha+1}$.
	
	Hence, we have $D\in\text{COMP}^{\text{ITRM}}_{\alpha}\setminus L_{\alpha+1}$, so COMP$^{\text{ITRM}}_{\alpha}\neq L_{\alpha+1}\cap\mathfrak{P}(\alpha)$.
\end{proof}

This theorem is a bit of an understatement. Using the truth predicate for $L_{\alpha}$, one can compute a code for $L_{\alpha+1}$. From this in turn, one obtains a code for $L_{\alpha+2}$ and so on. Thus, we actually get the following:

\begin{corollary}{\label{up to next limit}}
	If $\alpha$ is exponentially closed and $L_{\alpha}\not\models$ZF$^{-}$, then 
	COMP$^{\text{ITRM}}_{\alpha}\supseteq\mathfrak{P}(\omega)\cap L_{\alpha+\omega}$.
\end{corollary}

We will considerably extend this in the next section.

\bigskip

In any case, the proof of Theorem \ref{zf- and itrm} is now finished: For exponentially closed $\alpha$, Theorem \ref{left to right} shows that $L_{\alpha}\models$ZF$^{-}$ implies COMP$^{\text{ITRM}}_{\alpha}=L_{\alpha+1}\cap\mathfrak{P}(\alpha)$ and Theorem \ref{no zf- to more than succ} shows that $L_{\alpha}\not\models$ZF$^{-}$ implies that 
COMP$^{\text{ITRM}}_{\alpha}\neq L_{\alpha+1}\cap\mathfrak{P}(\alpha)$, which yields the desired equivalence.

\section{Towards the general Case}

We will now consider the case of ordinals $\alpha$ such that $L_{\alpha}\not\models$ZF$^{-}$. Although we are not able to determine the comptuational strength of $\alpha$-ITRMs in general, we give some information that should be helpful. The content of this section are generalizations of those obtained in \cite{CFKMNW} for ITRMs (i.e., the case $\alpha=\omega$).\footnote{We point out that our contribution to \cite{CFKMNW} was the lost melody theorem for ITRMs. In particular, we had no part in the proof that clockability implies computability for ITRMs, which is generalized to $\alpha$-ITRMs in Theorem \ref{clock implies write} below.} 

For the rest of the section, let $\alpha$ be an ordinal that is exponentially closed and ITRM-singular. 

We recall some standard terminology: For an ordinal $\gamma$, we say that $\gamma$ is $\alpha$-ITRM-computable if and only if it has an $\alpha$-ITRM-computable $\alpha$-code (i.e. a subset of $\alpha$ that codes it). We say that $\gamma$ is $\alpha$-ITRM-clockable if and only if there are an $\alpha$-ITRM-program $P$ and an ordinal $\zeta<\alpha$ such that $P(\zeta)$ halts in exactly $\gamma$ many steps.

\begin{lemma}{\label{appearance time bound}} [Cf. \cite{CFKMNW}, Theorem $5$]
	Let $P$ be an $\alpha$-ITRM-program, and let $\zeta<\alpha$ such that $P(\zeta)$ halts. Then no configuration appears more than $\omega^\omega$ many times in the computation of $P(\zeta)$.
\end{lemma}
\begin{proof}
This is proved by a generalization of the proof of Theorem $5$ of \cite{CFKMNW}. The new feature that one needs to take into account is the possibility of proper limits, i.e., that inferior limits are reached not by appearing cofinally often before, but as limits of increasing sequences from below. Since the general case requires a bit more care in some places, we elaborate the proofs a bit further than it is done in \cite{CFKMNW}.

So suppose for a contradiction that $P(\zeta)$ halts, but some configuration appears  $\geq\omega^\omega$ many times during this computation. The possible configurations are partially ordered by the component-wise $\leq$-relation. As configurations are finite tuples of ordinals, this ordering is well-founded. Let us assume without loss of generality that $c$ is minimal among the configurations appearing $\geq\omega^\omega$ many times during the computation of $P(\zeta)$.
Also, for $\iota<\omega^\omega$, let us denote by $\tau_{\iota}$ the $\iota$th time at which $c$ appears in the computation of $P(\zeta)$, and let $\tau:=\text{sup}_{\iota<\omega^\omega}\tau_\iota$.

\medskip

\textbf{Claim}: Between times $\tau_{0}$ and $\tau$, no component of any configuration occuring in the computation of $P(\zeta)$ was below the corresponing component in $c$.

\medskip

\begin{proof}
	Suppose otherwise. Thus, there is $\xi\in(\tau_{0},\tau)$ such that, at time $\xi$, some component of the current configuration $d$ was smaller than the corresponding component in $c$. Let $\delta$ be such that $\tau_{0}+\delta=\xi$. 
	
	Now, clearly, the same will happen at time $\tau_{\iota}+\delta$ for any $\iota<\omega^\omega$; moreover, if at least $\omega^\omega$ many occurences of $c$ would happen between times $\tau_{\iota}$ and $\tau_{\iota+\delta}$, then the same was true between times $\tau_{0}$ and $\xi$, contradicting the assumption that $\xi<\tau$. 
	
	Let us form a sequence $(\beta_{i}:i\in\omega)$ of ordinals as follows: $\beta_{0}=\tau_{0}$, $\beta_{2n+2}$ is the smallest ordinal of the form $\tau_{\iota^{\prime}}$ greater than $\beta_{2n+1}$ and $\beta_{2n+1}=\beta_{2n}+\delta$.
	
	It is not hard to see that the supremum $\hat{\gamma}$ of this sequence will be strictly below $\tau$. Let $\hat{\delta}$ be such that $\hat{\gamma}=\tau_{0}+\hat{\delta}$. Moreover, at time $\hat{\gamma}$, the configuration will be an inferior limit of a sequence of configurations that contains cofinally often both times at which the configuration was $c$ and at which it was $d$. Thus, at time $\hat{\gamma}$, the configuration $d^{\prime}$ will be strictly smaller than $c$ in the component-wise ordering. As $c^{\prime}$ appears at any time of the form $\tau_{\iota}+\hat{\delta}$ with $\iota<\omega^{\omega}$, this happens $\omega^{\omega}$ many times, contradicting the minimality of $c$.
\end{proof}

By the claim, the configuration at time $\tau$ is $c$, and we have a strong loop between times $\tau_{0}$ and $\tau$. But then, $P(\zeta)$ does not halt, a contradiction.

\end{proof}

\begin{remark}
Note that the bound $\omega^\omega$ is optimal in the above result. To this end, recall (e.g., from \cite{CFKMNW}) that a 'flag' for an ITRM consists of two registers that initially contain $0$ and $1$ and swap their contents in each step, so that the first time at which they will have equal contents will be $\omega$. By nesting flags (which is used in \cite{CFKMNW} to show that ordinals $<\omega^{\omega}$ are ITRM-clockable), it is easy to obtain, for any $k\in\omega$ and any $\alpha>0$, a halting $\alpha$-ITRM-program that repeats some configuration at least $\omega^{k}$ many times, and in fact with a program that only generates register contents $0$ and $1$.
\end{remark}

From the proof of Theorem \ref{appearance time bound}, we further obtain the following observation:

\begin{corollary}{\label{limit loops}}
	Let $\delta$ be a limit ordinal, and suppose that in the computation of $P(\zeta)$, the configuration $c$ appears both at time $\delta$ and unboundedly often before $\delta$. Then $P(\zeta)$ is looping.
\end{corollary}

From now on, the proofs from \cite{CFKMNW} apply verbatim in the general context. We thus restrict ourselves to giving the results in their general form and sketching the proofs for the convenience of the reader.

The following lemma will be needed below for $\iota<\omega^{\omega}$.

\begin{lemma}{\label{clock low}} [Cf. \cite{CFKMNW}, Lemma $4$]
	Every ordinal $\iota<\omega_{\omega}^{\text{CK}}$ is $\alpha$-ITRM-clockable. 
	Moreover, there is a recursive function $f$ that sends (an encoding of the Cantor normal forms of) each ordinal $\gamma$ in $\omega^\omega$ to an $\alpha$-ITRM-program $P$ such that $P$ clocks $\gamma$.
\end{lemma}
\begin{proof}
	Let $\iota<\omega_{\omega}^{\text{CK}}$.
	By the results of \cite{CFKMNW} and \cite{K1},
	$\iota$ is ITRM-clockable. Clearly, when $\alpha\geq\omega$, then an $\alpha$-ITRM can simulate an ITRM (using an extra register which initially contains $\omega$) without time lapse. Hence $\iota$ is $\alpha$-ITRM-clockable.
	
	The second statement follows similarly from Proposition $1$ of \cite{CFKMNW}, the proof of which clearly yields such a recursive function.
\end{proof}

\begin{thm}{\label{clock implies write}} [Cf. \cite{CFKMNW}, Theorem $7$]
	If $\beta$ is $\alpha$-ITRM-clockable, then $\beta$ is $\alpha$-ITRM-computable. Moreover, the suprema of the $\alpha$-ITRM-clockable and the $\alpha$-ITRM-computable ordinals coincide.
\end{thm}
\begin{proof}
Let $P$ and $\zeta$ be such that $P(\zeta)$ runs for exactly $\beta$ many steps. Suppose that $P$ uses $n\in\omega$ many registers. Pick some effective bijection $f:\alpha\rightarrow\alpha^{n}\times\omega\times\omega^{\omega}$. An $\alpha$-code $c$ is obtained as follows: We let $p(\iota,\xi)\in c$ if and only if, with $f(\iota)=(a_{1},...,a_{n},j,\iota^{\prime})$ and $f(\xi)=(b_{1},...,b_{n},k,\iota^{\prime\prime})$, the $\iota^{\prime}$-th occurence of the configuration $(a_{1},...,a_{n},j)$ preceeds the $\iota^{\prime\prime}$-th occurence of the configuration $(b_{1},...,b_{n},k)$ in the computation of $P(\zeta)$ (and both occurences exist in this computation). 

This can be computed as follows: By Lemma \ref{clock low}, compute a program $Q_{0}$ that clocks $\iota^{\prime}$ and a program $Q_{1}$ that clocks $\iota^{\prime\prime}$. 
Then run $P(\zeta)$. Whenever $(a_{1},...,a_{n},j)$ occurs, run one step of $Q_{0}$ in some separate registers, and likewise for $(b_{1},...,b_{n},k)$ and $Q_{1}$. If $Q_{0}$ halts before $Q_{1}$ in this computation, return `yes', otherwise (i.e., if the computation of $P(\zeta)$ halts and no such occurences were detected or if $Q_{1}$ halts before $Q_{0}$), return `no'.

Concerning the suprema, the above yields that the supremum of the $\alpha$-ITRM-computable ordinals is not smaller than that of the $\alpha$-ITRM-clockable ordinals. For the converse, just note that, if $\alpha$ is ITRM-singular, then we can run a depth-first search on an $\alpha$-code $c$ for an ordinal $\beta$ to test it for well-foundedness, which takes at least $\beta$ many steps.

\end{proof}

\begin{lemma}{\label{no gaps}} [Cf. \cite{CFKMNW}, Lemma $3$]
	The set of $\alpha$-ITRM-clockable ordinals is downwards closed: That is, if $\xi$ is $\alpha$-ITRM-clockable and $\iota<\xi$, then $\iota$ is also $\alpha$-ITRM-clockable.
\end{lemma}
\begin{proof}
Let $P$ be an $\alpha$-ITRM-program and $\zeta$ an ordinal such that $P(\zeta)$ clocks $\xi$. At time $\iota$ in the computation of $P(\zeta)$, some configuration $c$ appears for the $\gamma$-th time. By Theorem \ref{appearance time bound}, we have $\gamma<\omega^\omega$. By Lemma \ref{clock low}, $\gamma$ is $\alpha$-ITRM-clockable, say by the program $Q$. Now run $P(\zeta)$ with $c$ stored in some extra registers and run one step of $Q$ whenever $c$ occurs in that computation. When $Q$ halts, halt. This clocks $\iota$.
\end{proof}

Let us from now on denote by $\beta(\alpha)$ the supremum of the $\alpha$-ITRM-clockable ordinals.

\begin{corollary}{\label{early looping}}
[Cf. \cite{CFKMNW}]	Let $P$ be an $\alpha$-ITRM-program, and let $\zeta<\alpha$. Suppose that $P(\zeta)$ does not halt. Then there are ordinals $\iota<\xi<\beta(\alpha)$ such that $(\iota,\xi)$ witnesses the looping of $P(\zeta)$. Moreover, $\beta(\alpha)$ is minimal with this property.
\end{corollary}
\begin{proof} 
	Suppose that $P$ uses $n$ registers.
	Since $P(\zeta)$ does not halt and all configurations come from the set $\alpha^{n}\times\omega$, some configuration must occur more than $\omega^\omega$ many times, so Lemma \ref{appearance time bound} implies that the computation will eventually enter a strong loop. Say that $(\iota,\xi)$ is lexically minimal such that $(\iota,\xi)$ witnesses the looping of $P(\zeta)$. Let $c$ be the configuration at time $\iota$. 
	
	We claim that $\xi$ is $\alpha$-ITRM-clockable, which implies $\xi<\beta(\alpha)$, as desired. 
	
	To clock $\xi$, let us run $P(\zeta)$ with $c$ stored in some extra registers. Moreover, we use a further `index' register $R$, which initially contains $0$. Whenever $c$ is assumed during the computation of $P(\zeta)$, the content of $R$ is changed to $1$. When a further configuration is $<c$ in any component, the content of $R$ is changed back to $0$. On the other hand, if $c$ reoccurs with $R$ containing $1$, we halt. 
	
	Minimality is clear, as we can simply run a non-halting computation after a halting computation.
\end{proof}

For the next result, it will be important that one can `read out' ordinals from codes for ordinals. More specifically:

\begin{lemma}{\label{read out codes}}
Let $c\subseteq\alpha$ be an $\alpha$-code for an ordinal $\gamma<\alpha$ obtained from a program $C$ clocking $\gamma$ (possibly with parameters) as described in the proof of Theorem \ref{clock implies write}. Then there are $\alpha$-ITRM-programs $P$ and $Q$ such that:

\begin{enumerate}
	\item For any $\iota<\gamma$, $P(\iota)$ halts and outputs an ordinal $\xi<\alpha$ such that $\xi$ represents $\iota$ in $c$.
	\item For any $\iota<\alpha$, $Q(\iota)$ halts and outputs an ordinal $\xi<\alpha$ such that $\iota$ represents $\xi$ in $c$.
\end{enumerate}
\end{lemma}
\begin{proof}
First, observe that, since $\gamma<\alpha$, no configuration can occur in the computation of $C$ more than $\alpha$ many times.
\begin{enumerate}
\item The program in question works as follows: Let $\rho<\gamma$ be given in some input register. Run $C$. Along with $C$, count upwards in a separate register until $\rho$ is reached. At this point, we have determined the $\rho$-th configuration in the computation of $C$, say $d$. Now run $C$ again for $\rho$ many steps in this way, this time counting upwards in some register whenever the configuration $d$ occurs during the computation. Let $\delta$ be the content of that register when $C$ reaches its $\rho$th step. (In particular, we have $\delta<\omega^{\omega}$ and $\delta<\gamma<\alpha$.) Now, knowing $d$ and $\delta$, we can compute the ordinal representing $\rho$ in the sense of $c$ by simply computing the tuple code.
\item Let $\rho$ be given in an input register; we want to determine the ordinal $\xi<\gamma$ represented by $\rho$ in $c$. To this end, count upwards in some register up to $\gamma$. For every $\iota<\gamma$, uses the program from (1) to compute the ordinal $\xi(\iota)$ represented by $\iota$ in $c$. At some point, we will have $\xi(\iota)=\rho$; when this happens, output $\iota$.
\end{enumerate}

\end{proof}

\begin{thm}{\label{strength and L}}
	For every exponentially closed and ITRM-singular ordinal $\alpha$ 
	and every $x\subseteq\alpha$, $x$ is $\alpha$-ITRM-computable if and only if $x\in L_{\beta(\alpha)}$. 
\end{thm}
\begin{proof}
           Suppose that $x\subseteq\alpha$ is $\alpha$-ITRM-computable. Thus, there are an $\alpha$-ITRM-program $P$ and an ordinal $\zeta$ such that, for any $\iota<\alpha$, we have $P(\iota,\zeta)\downarrow=1$ if and only if $\iota\in x$ and otherwise $P(\iota,\zeta)\downarrow=0$. In particular, $P(\iota,\zeta)$ halts for every $\iota<\alpha$.
          Now consider the program $Q$ that counts upwards in some register $R$ starting with $0$ and runs $P(\iota,\zeta)$ for every $\iota$ appearing in that register until $R$ overflows. This program will terminate, say in $\gamma$ many steps. In particular, we will have $\gamma<\beta(\alpha)$ and $\gamma$ will be larger than the halting time of $P(\iota,\zeta)$ for 
          each $\iota<\alpha$. So $x$ is definable over $L_{\gamma}$, thus $x\in L_{\gamma+1}\subseteq L_{\beta(\alpha)}$.

           It remains to show that, if an ordinal $\gamma$ is $\alpha$-ITRM-computable and $x\in\mathfrak{P}(\alpha)\cap L_{\gamma}$, 
	then $x$ is $\alpha$-ITRM-computable. 
	
	We sketch the construction, which is based on the way Koepke et al. used to show that Ordinal Register Machines (ORMs) compute all constructible sets of ordinals.
	
          Elements of $L_{\gamma}$ can be `named' by triples of the form $(\beta,\phi,\xi)$, 
	where $\xi<\beta<\alpha$ are ordinals and $\phi$ is an $\in$-formula. 
	Here, $(\beta,\phi,\xi)$ will represent the set 
	$\{x\in L_{\beta}:L_{\beta}\models\phi(\xi)\}$.
	Clearly (by exponential closure of $\alpha$), names can be encoded as ordinals in $\alpha$ in way that allows us to 
	compute codes from their components and components from codes on an $\alpha$-ITRM. 
	Let us fix such an encoding $f:\alpha\rightarrow\alpha\times\omega\times\alpha$.
	
	We can now use a recursive truth predicate algorithm in the spirit of Lemma \ref{singular truth predicate evaluation} to determine, for codes $\iota$ and $\xi$, whether or $f(\iota)\in f(\xi)$. We then obtain an $\alpha$-code for $L_{\gamma}$ as the set of  $p(\iota,\xi)$ such that $f(\iota)\in f(\xi)$.
	
	From this code, we can read out all elements of $L_{\gamma}$ using Lemma \ref{read out codes}.

	
	

\end{proof}

Thus, in contrast to Theorem \ref{left to right}, which spoils the general picture of ordinal computability that `computational strength corresponds to the $L$-level indexed by the supremum of the clockable ordinals' which holds for all models studied so far (including wITRMs (\cite{K}), ITRMs (\cite{}, \cite{K1}), ITTMs (\cite{HL}), $\alpha$-ITTMs (\cite{K1}, \cite{C}, \cite{alpha itrms}), ORMs (\cite{ORM}) and OTMs (\cite{OTM}) with and without parameters, the `hypermachines' of Friedman and Welch (\cite{FW}) and infinite time Blum-Shub-Smale machines (\cite{KS1},\cite{KM}), for ITRM-singular $\alpha$, all is well again.\footnote{We note that this is not the only place in ordinal computability where this happens: In fact, for Ordinal Turing Machines (OTMs), let $\rho$ be the supremum of the ordinals with eventually OTM-writable real codes and let $\eta$ be the supremum of the stabilization times for real numbers. Since $\rho$ is a supremum of a constructibly countable set of constructibly countable ordinals, we have $\rho<\omega_{1}^{L}$; on the other hand, J. Hamkins observed (see \cite{HMO}) that there are stabilization times way over $\omega_{1}^{L}$, so that $\eta>\omega_{1}^{L}$; in particular, $\eta$ is way larger than $\rho$.}

\subsection{$(\alpha,\beta)$-ITRMs}

In \cite{K1}, Koepke defined $(\alpha,\beta)$-(w)ITRMs, which are $\alpha$-(w)ITRMs with time restricted to ordinals $<\beta$. 

\begin{defini}{\label{alpha beta comp def}}
	A set $x\subseteq\alpha$ is $(\alpha,\beta)$-(w)ITRM-computable if and only if there are an $\alpha$-(w)ITRM-program $P$ and some $\xi,\rho<\alpha$ such that, for all $\iota<\alpha$, $P(\iota,\xi)$ halts in $<\rho$ many steps with output $\chi_{x}(\iota)$.
	
	We denote the set of $(\alpha,\beta)$-ITRM-computable subsets of $\alpha$ by COMP$^{\text{ITRM}}_{(\alpha,\beta)}$ and the set of $(\alpha,\beta)$-wITRM-computable subsets of $\alpha$ by COMP$^{\text{wITRM}}_{(\alpha,\rho)}$.
\end{defini}

Combining the above proof with a more careful analysis of running times, we obtain the computational strength of $(\alpha,\beta)$-ITRMs, 
 partially\footnote{Note that, as exponential closure is a rather weak requirement on running times, this can be regarded as the major part of the question, possibly covering all interesting cases.} answering a question of Koepke, see \cite{K1}, p. 6.

\begin{thm}{\label{alpha beta ITRM}}
	Let $\rho\in(\alpha,\beta(\alpha)]$ be exponentially closed. Then COMP$^{\text{ITRM}}_{(\alpha,\rho)}=L_{\rho}\cap\mathfrak{P}(\alpha)$.
\end{thm}
\begin{proof}
Clearly, if $x$ is computable by an $\alpha$-ITRM with time bounded by $\gamma<\rho$, then $x\in L_{\gamma+\omega}\subseteq L_{\rho}$.

On the other hand, suppose that $x\in L_{\rho}$. Thus, there is $\gamma<\rho$ such that $x\in L_{\gamma}$. As $\gamma<\rho<\beta(\alpha)$, $\gamma$ is $\alpha$-ITRM-clockable.
By the proof of Theorem \ref{clock implies write}, a code for $\gamma$ is $\alpha$-ITRM-computable with time bound $\gamma\cdot 2$.

Now, $x$ is of the form $\{\iota<\alpha:L_{\delta}\models\phi(\iota,\zeta)\}$ for some $\delta<\gamma$, some ordinal parameter $\zeta<\delta$ and some $\in$-formula $\phi$. Thus, $x$ is `named' by the tripel $(\delta,\phi,\zeta)$.

To determine whether $\iota\in x$ for some $\iota<\alpha$, one now again uses the bounded truth predicate algorithm for evaluating whether $L_{\delta}\models\phi(\iota,\zeta)$. It is now easy to see (see, e.g., \cite{KS1}, p. 314) that the running time of this algorith will be bounded below the next exponentially closed ordinal after $\text{max}\{\alpha,\delta,\zeta\}$, and thus in particular below $\rho$, as required. Thus $x\in$COMP$^{\text{ITRM}}_{(\alpha,\beta)}$. 
\end{proof}

\subsection{Properties of $\beta(\alpha)$}

What remains to be done in order to determine the computational strength of $\alpha$-ITRMs when $L_{\alpha}\not\models$ZF$^{-}$ is to determine $\beta(\alpha)$. In this section, we will give upper bounds and, in some special cases, lower bounds on the value of $\beta(\alpha)$; moreover, we will prove some properties of $\beta(\alpha)$.

Our first observation is that Corollary \ref{limit loops} yields $\Pi_{3}$-reflecting ordinals as bounds on halting times (we denote by COMP$^{\text{wITRM}}_{\alpha}$ the set of $\alpha$-wITRM-decidable subsets of $\alpha$); we remark that $\Pi_{3}$-reflecting ordinals as bounds on halting times for a strengthened type of Infinite Time Blum-Shub-Smale-machines working on natural or rational numbers are mentioned by Welch in \cite{W1}; the proof below probably bears some similarity to his (unpublished) argument, which is not known to us. We received a further hint at considering $\Pi_{3}$-reflecting ordinals from \cite{M}, according to which the first $\Pi_{3}$-reflecting ordinal is way smaller than the first $\Sigma_{2}$-admissible ordinal. As every $\Sigma_{2}$-admissible is $\Pi_{3}$-reflecting but not vice versa, part (1) of the following theorem improves Theorem \ref{sigma2 admissible wITRM strength} from the introduction.

\begin{thm}{\label{pi3 reflecting}}

We have the following statements:

\begin{enumerate}
		\item If $\alpha$ is $\Pi_{3}$-reflecting, then COMP$^{\text{wITRM}}_{\alpha}=\Delta_{1}(L_{\alpha})\cap\mathfrak{P}(\alpha)$.
		\item If $\gamma$ is the smallest $\Pi_{3}$-reflecting ordinal $>\alpha$, then $\beta(\alpha)<\gamma$.
	\end{enumerate}
\end{thm}
\begin{proof}
(1) Suppose that $\alpha$ is $\Pi_{3}$-reflecting, and let $P$ be an $\alpha$-wITRM-program, $\zeta<\alpha$. Suppose that $P(\zeta)$ does not halt before time $\alpha$, and let $c=(l,\rho_{1},...,\rho_{n})$ be the configuration of $P(\zeta)$ at time $\alpha$. We will show that $c$ has appeared unboundedly often before time $\alpha$, from which it follows by Corollary \ref{limit loops} that $P(\zeta)$ does not halt at all.

That the $i$th register contains $\rho_{i}$ at time $\alpha$ means that (i) for every $\iota<\alpha$, there is $\xi\in(\iota,\alpha)$ such that, at time $\xi$, the content of the $i$th register was at most $\rho_{i}$ and (ii) that, for every $\iota<\rho_{i}$, there is $\xi<\alpha$ such that, for all $\gamma\in(\xi,\alpha)$, the $i$-th register had a content $>\iota$ at time $\gamma$. This is clearly a $\Pi_{3}$-statement that holds in $L_{\alpha}$. Since $\alpha$ is $\Pi_{3}$-reflecting, it holds at some earlier ordinal $\bar{\alpha}$. The same is true when one additionally demands in both clauses that $\xi$ is larger than some given bound $\delta<\alpha$. The active program line $l$ can be dealt with in the same way (it can be regarded as another register that only contains natural numbers below a given bound). Thus, for every $\delta<\alpha$, there is a time $\tau\in(\delta,\alpha)$ such that $c$ appeared at time $\tau$, i.e. $c$ appeared unboundedly often before time $\alpha$.

However, by Lemma \ref{koepke seyfferth}, 
one sees that register machines with time and and space bound $\alpha$ compute exactly those subsets of $\alpha$ contained in $\Delta_{1}(L_{\alpha})\cap\mathfrak{P}(\alpha)$.
\bigskip

(2) The proof that $\beta(\alpha)\leq\gamma$ works by a similar argument; the only modification is that, when some $\rho_{i}$ is equal to $0$, one needs to distinguish whether or not this is due to an overflow when expressing this as a $\Pi_{3}$-formula. Now, the statement that, for all $k\in\omega$ and all $\zeta\in\alpha$, there is either a halting computation of $P_{k}(\zeta)$ or a partial computation with a strong loop is $\Pi_{2}$ and holds in $L_{\gamma}$, as it holds in $L_{\beta(\alpha)}$ and $\gamma\geq\beta(\alpha)$. Since $\gamma$ is $\Pi_{3}$-reflecting, it follows that there is $\bar{\gamma}<\gamma$ such that the same holds in $L_{\bar{\gamma}}$. Thus $\beta(\alpha)\leq\bar{\gamma}<\gamma$.
\end{proof}

As a consequence of the last part of the proof of Theorem \ref{pi3 reflecting}, we obtain:

\begin{corollary}{\label{no pi2 reflection}}
$\beta(\alpha)$ is not $\Pi_{2}$-reflecting. In particular, $\beta(\alpha)$ is not admissible.
\end{corollary}

\begin{prop}{\label{no overflow at beta}}
Let $P$ be an $\alpha$-ITRM-program, $\zeta<\alpha$. Then, at time $\beta(\alpha)$, no registers of the computation of $P(\zeta)$ overflows.
\end{prop}
\begin{proof}
If $P(\zeta)$ halts, it does so before time $\beta(\alpha)$, so there is no overflow at time $\beta(\alpha)$. Now suppose that $P(\zeta)$ does not halt. As $\beta(\alpha)$ is the supremum of the looping times for $\alpha$-ITRMs, we know that $P(\zeta)$ entered a strong loop before time $\beta(\alpha)$, which, by additive indecomposability of $\beta(\alpha)$, has been repeated unboundedly often below $\beta(\alpha)$. In particular, there is a configuration $c$ that appeared unboundedly often before time $\beta(\alpha)$. But then, by the liminf-rule, there cannot be an overflow at time $\beta(\alpha)$.
\end{proof}

We mention some further properties of $\beta(\alpha)$. 

%
%

We also note the following rather straightforward generalization of the induction used in Koepke \cite{K1} for ITRMs:

\begin{defini}{\label{alpha safe def}}
Let us say that an ordinal $\tau$ is $\alpha$-safe if there is no $\alpha$-ITRM-program $P$ and no $\zeta<\alpha$ such that $(P,\zeta)$ has a proper limit at time $\tau$.

Let us denote the $i$th ordinal which is both additively indecomposable and $\alpha$-safe by $\tau_{i}(\alpha)$ and let $\tau_{\omega}(\alpha):=\text{sup}_{i\in\omega}\tau_{i}(\alpha)$. 
Moreover, define $\tau_{i}^{w}(\alpha)$ in the analogous way for $\alpha$-wITRMs.
\end{defini}

\begin{lemma}{\label{alpha induction}}
For all exponentially closed $\alpha$, we have $\beta(\alpha)\leq\tau_{\omega}(\alpha)$. Moreover, the halting times of $\alpha$-wITRMs are bounded above by $\tau_{1}^{w}(\alpha)$.
\end{lemma}
\begin{proof}
Let $P$ use the registers $R_{1},...,R_{n}$. 
As in \cite{K1}, we really show the following statement:

\bigskip
\textbf{Claim}: If less than $k$ many registers contain $0$ at time $\tau_{k}(\alpha)$ in the computation of $P(\zeta)$, then $P(\zeta)$ loops.

\bigskip
To see this, first suppose that, at time $\tau_{1}(\alpha)$, no register contains $0$. Let $c=(l,r_{1},...,r_{n})$ be the configuration of $P(\zeta)$ at time $\tau_{1}(\alpha)$. 
By definition of $\tau_{1}(\alpha)$, there is $\tau<\tau_{1}(\alpha)$ such that the content of the $i$th register does not drop below $r_{i}$ between times $\tau$ and $\tau_{1}(\alpha)$ (and similarly for the active program line). Consequently, $r_{i}$ has occured as the content if the $i$th register unboundedly often before time $\tau_{1}(\alpha)$.  If we can show that $c$ occured between times $\tau$ and $\tau_{1}(\alpha)$, we have a strong loop for $P(\zeta)$, as desired. We proceed as follows: As $\tau$ occurs below the looping or halting time of $P(\zeta)$, $\tau$ is clockable. Thus, we can run $P(\zeta)$ for $\tau$ many steps. After that, we continue to run $P(\zeta)$, but along with that, we have a new registers $R$ (initially containing $0$) and run a routine that works in phases and observes the computation of $P(\zeta)$ do detect the following: In phase $i\in\{1,2,...,(n-1)\}$, it waits for $R_{i}$ to contain $r_{i}$. When this happens, it counts $1$ upwards in $R$ and switches to phase $(i+1)$. In phase $n$, it waits for $R_{n}$ to contain $r_{n}$ and, when that happens, counts $1$ upwards in $R$ and switches to phase $0$. In phase $0$, it waits for the active program line to be $l$ and when that happens, it counts $1$ upwards in $R$ and switches to phase $1$. Thus, when $R$ contains a limit ordinal, the configuration of $P(\zeta)$ is $c$. If that happens at $\tau_{1}(\alpha)$ for the first time, the content of $R$ at time $\tau_{1}(\alpha)$ is $\omega$, while it was $<\omega$ at all earlier times, contradicting the definition of $\tau_{1}(\alpha)$. Thus, $c$ occurs between times $\tau$ and $\tau_{1}(\alpha)$, so $P(\zeta)$ is indeed looping.

\bigskip

The inductive step now works as in the proof of Theorem \ref{halting time bound} above or as the proof of the main theorem in \cite{K1}:
Suppose that the statement is proved for $k$ and suppose that, at time $\tau_{k+1}(\alpha)$, at most $k$ registers contain $0$. If none of these $0$s is due to an overflow, we are back in the situation of the base case. Otherwise, suppose that $R_{1}$ overflows at time $\tau_{k+1}(\alpha)$. Again, there is $\tau<\tau_{k+1}(\alpha)$ such that, from time $\tau$ on, the content of $R_{i}$ is always $> 0$; in particular, no register that does not contain $0$ at time $\tau_{k+1}(\alpha)$ contains $0$ after time $\tau$ (up to time $\tau_{k+1}(\alpha)$). Now, as $R_{1}$ overflows, $R_{1}$ does not contain $0$ from some time $\xi<\tau_{k+1}(\alpha)$ on; let $\gamma=\text{max}\{\xi,\tau\}$. By additive indecomposability of $\tau_{k+1}(\alpha)$, we have $\gamma+\tau_{k}(\alpha)<\tau_{k+1}(\alpha)$. But then, at time $\gamma+\tau_{k}(\alpha)$, less than $k$ registers contained $0$. As we can regard this as the $\tau_{k}(\alpha)$th step of the computation of $P$ starting in the configuration that $P(\zeta)$ had at time $\gamma$, it follows by induction that $P(\zeta)$ is looping.

\bigskip

The argument for the second statement is basically the base case of the above induction.
\end{proof}

\begin{remark}
	It is easy to see that the supremum of the $\alpha$-wITRM-clockable ordinals is both additively indecomposable and $\alpha$-safe for weak machines. Thus, this supremum is in fact equal to $\tau_{1}^{w}(\alpha)$. 
\end{remark}



We saw above that, for ITRM-singular $\alpha$ (i.e., for $L_{\alpha}\not\models$ZF$^{-}$), the depth-first search algorithm testing subsets of $\alpha$ for coding well-orderings can be performed on an $\alpha$-ITRM. This algorithm has the property that it produces an infinite descending sequence in the case of an ill-founded input. Consequently, denoting by WO$_{\alpha}$ the set of subsets of $\alpha$ coding well-orderings, we have the following property of $\beta(\alpha)$:

\begin{corollary}{\label{wo true}}
For all $c\in\mathfrak{P}(\alpha)\cap L_{\beta(\alpha)}$, we have that $c\in$WO$_{\alpha}$ if and only if $L_{\beta(\alpha)}\models c\in$WO$_{\alpha}$.
\end{corollary}

In particular, this means that $L_{\beta(\alpha)}$ is $\Pi_{1}^{1}$-true, i.e. for any $\Pi_{1}^{1}$-formula $\phi$ with parameters in $L_{\beta(\alpha)}$, we have $L_{\beta(\alpha)}\models\phi$ if and only if $\phi$ holds in $V$.\footnote{We suspect that this can be generalized via generalized descriptive set theory to a more general version of $\Pi_{1}^{1}$-statements referring to elements and subsets of $\alpha$.} 

\subsection{Special Cases}

With further conditions on $\alpha$, we can obtain more precise estimates of $\beta(\alpha)$. 

Recall that an ordinal $\alpha$ is an index if and only if $(L_{\alpha+1}\setminus L_{\alpha})\cap\mathfrak{P}(\omega)\neq\emptyset$. Moreover, for an ordinal $\alpha$ and a natural number $i$, $\alpha^{+}$ denotes the next admissible ordinal after $\alpha$, $\alpha^{+i}$ denotes the $i$th admissible ordinal after $\alpha$ and $\alpha^{+\omega}$ denotes the next limit of admissible ordinals after $\alpha$. Note that, if $\alpha$ is an index, then the comprehension axiom for subsets of $\omega$ does not hold in $L_{\alpha}$, so $L_{\alpha}\not\models\text{ZF}^{-}$ and thus $\alpha$ is ITRM-singular.

\begin{prop}{\label{beta alpha properties}}
	$\beta(\alpha)$ is not admissible. If $\alpha$ is an index, then $\beta(\alpha)$ is a limit of admissible ordinals.
\end{prop}
\begin{proof}
	That $\beta(\alpha)$ is not admissible was seen in Corollary \ref{no pi2 reflection} above.\footnote{For a different way to see that $\beta(\alpha)$ is not admissible, note that $\beta(\alpha)>\alpha$ and the function that maps $(k,\zeta)\in\omega\times\alpha$ to the halting- or looping time (the end of the first repetition of a strong loop) of the $k$th program in the input $\zeta$ is $\Sigma_{1}$ (in fact $\Delta_{1}$) over $L_{\beta(\alpha)}$.}
	
	Now suppose that $\alpha$ is an index. Let $\gamma<\beta(\alpha)$. Thus, $\gamma$ is clockable, so $\gamma+1$ is an index. To see this, note that a real number in $L_{\alpha+1}\setminus L_{\alpha}$ can be used in the definition of a real number over $L_{\gamma}$. Now suppose that $P(\iota)$ runs for exactly $\gamma$ many steps. There is a real $r$ in $L_{\alpha}$ that codes $\gamma$ by assumption. Now $\gamma+1$ is minimal with the property that $L_{\gamma+1}$ believes that there is an ordinal $\rho$ coded by $r$ such that $P(\rho)$ halts. Hence, the $\Sigma_{1}$-Skolem hull of $\{r\}$ in $L_{\gamma+1}$ is $L_{\gamma+1}$, so $\gamma+1$ is a index by standard finestructure.

	Let $g\in L_{\gamma+2}\setminus L_{\gamma+1}$ be a code for $\gamma$. Then $g$ is $\alpha$-ITRM-computable. The same holds for all ordinals that are recursive in $g$; as $\alpha$-ITRMs can simulate ITRMs, we can check all of these for well-foundedness and compute a code for their sum, which will be $\omega_{1}^{\text{CK},g}$.
	Thus, we obtain $\gamma<\omega_{1}^{\text{CK},g}<\beta(\alpha)$, so there is an admissible clockable ordinal above $\gamma$, as desired.
\end{proof}

\begin{remark}
	Note that this has the consequence that the computational strength of $\alpha$-ITRMs may make wild jumps as the number of used registers increases: In fact, if $L_{\alpha}\models\Sigma_{n(k)}$-collection (with $k\in\omega$ and $n(k)$ as in Corollary \ref{refined halting bound}), an $\alpha$-ITRM with $\leq k$ registers will halt or loop in $<\alpha^{k+1}$ many steps, while the halting times with any number of registers go at least up to $\alpha^{+\omega}$. This condition is for example satisfied by the smallest $\alpha>\omega$ that satisfies $\Sigma_{n(k)}$-collection.
\end{remark}

Concerning lower bounds for the supremum of the clockable ordinals, this yields the following partial result. 

\begin{corollary}{\label{index ordinals up to limit of admissibles}}
	Let $\alpha$ be an index. Then $\beta(\alpha)\geq\alpha^{+\omega}$.
\end{corollary}
\begin{proof}
	We have that $\beta(\alpha)>\alpha$ and from Proposition \ref{beta alpha properties} it follows that $\beta(\alpha)$ is a limit of admissible ordinals. As $\alpha^{+\omega}$ is by definition the smallest ordinal with those properties, we have $\beta(\alpha)\geq\alpha^{+\omega}$.
	
\end{proof}

We recall from \cite{AS} that an admissible ordinal $\alpha$ is called \emph{Gandy} if and only if the supremum of the $\alpha$-recursive ordinals equals $\alpha^{+}$. Gostanian \cite{Go} obtained several sufficient criteria for a countable ordinal to be Gandy. Generalizations to the uncountable case were given by Abramson and Sacks \cite{AS}. 

We start by connecting Gandyness to $\alpha$-ITRMs.

\begin{lemma}{\label{gandy and beta}}
If $\beta$ is an admissible Gandy ordinal with $\beta\in(\alpha,\beta(\alpha)]$, then $\beta^{+}<\beta(\alpha)$.
\end{lemma}
\begin{proof}
	Suppose that $\beta\leq\beta(\alpha)$ is admissible. As we saw above, $\beta(\alpha)$ is not admissible. Thus, we actually have $\beta(\alpha)>\beta$. It follows that $\beta+1$ is $\alpha$-ITRM-clockable. Consequently, every subset of $\alpha$ contained in $L_{\beta+1}$ is $\alpha$-ITRM-computable, which includes every $\alpha$-recursive subset of $\alpha$. Thus, we have $\beta^{+}\leq\beta(\alpha)$. By inadmissibility of $\beta(\alpha)$ again, we have $\beta^{+}<\beta(\alpha)$.
\end{proof}

We recall the following special case of a statement from Gostanian \cite{Go}, Corollary $2.1.1$:

\begin{lemma}{\label{gostanian}}[Cf. \cite{Go}, Corollary $2.1.1$]
	
	Let $\alpha$ be an admissible ordinal such that $\alpha$ is countable in $L_{\alpha^{+}}$. 
	Let $\phi(\vec{p})$ be an $\in$-formula with parameters in $L_{\alpha}$ and suppose that $\alpha$ is minimal with $L_{\alpha}\models\phi(\vec{p})$. Then $\alpha$ is Gandy. 
\end{lemma}

This allows us to show for many special cases of $\alpha$ that $\beta(\alpha)$ is either rather small or quite large:

\begin{thm}{\label{weak or very strong}}
	Let $\alpha$ be admissible and countable in $L_{\alpha^{+}}$.
	Then either $\beta(\alpha)<\alpha^{+}$ or $\beta(\alpha)\geq\alpha^{+\omega}$. 
\end{thm}
\begin{proof}
Suppose that the first alternative fails, i.e. $\beta(\alpha)\geq\alpha^{+}$. We now show inductively that $\beta(\alpha)\geq \alpha^{+i}$, for every $i\in\omega$, which implies that the second alternative holds. The base case has already been dealt with, so it remains to prove the induction step. Hence, let us assume that $\beta(\alpha)\geq \alpha^{+i}$; we will show that $\beta(\alpha)\geq\alpha^{+(i+1)}$. 

Since $\beta(\alpha)$ is not admissible, we have $\beta(\alpha)>\alpha^{+i}$. Hence $\alpha^{+i}$ is $\alpha$-ITRM-clockable. Thus, there exists a program $P$ and some parameter $\rho<\alpha$ such that $P(\rho)$ runs for exactly $\alpha^{+i}$ many steps. 

We will now express that fact that $P(\rho)$ halts by a formula $\phi$ that holds in $L_{\gamma}$ if and only if $\gamma\geq\alpha^{+i}$. (Note that ``there is a halting computation by $P(\rho)$ will not work, as $L_{\alpha^{+i}}$ will not contain a computation of length $\alpha^{+i}$.) Let $c=(l,r_{1},...,r_{n})$ be the halting configuration of the computation of $P(\rho)$, where $l\in\omega$ is the index of the active program line and $r_{1},...,r_{n}<\alpha$ are the register contents. Then $c\in L_{\alpha}$. Now let $n$ be the maximal index of a register used by $P$ and let $\phi(\rho,c)$ be the conjunction of the following statements:

\begin{itemize}
	\item There is $\tau$ such that every partial computation of $P(\rho)$ of length $>\tau$ has the active program line index $\geq l$ at all times $>\tau$.
	\item For every $\tau$ such that there is a partial computation of $P(\rho)$ of length $\tau$, there is $\tau^{\prime}>\tau$ and a partial computation $C$ of $P(\rho)$ of length $\tau^{\prime}+1$ such that, at time $\tau^{\prime}$, $C$ has the active program line index $l$. 
	\item (For every $i\leq n$.) For every $\rho<r_{i}$, there is $\tau$ such that there is a partial computation of $P(\rho)$ of length $\tau$ and every partial computation $C$ of $P(\rho)$ of length $>\tau$ has the content of the $i$th register $>\rho$ from time $\tau$ on. 
    \item (For every $i\leq n$.) For every $\tau$ such that there is a partial computation of $P(\rho)$ of length $\tau$, there is $\tau^{\prime}>\tau$ and a partial computation of $P(\rho)$ of length $\tau^{\prime}+1$ such that, at time $\tau^{\prime}$, the content of the $i$th register is $\leq r_{i}$. 
\end{itemize}

These statements simply encode the liminf-rule. The first two statements imply that, at time $\alpha^{+i}$, the active program line index is $l$, while the last two imply that the register contents are $(r_{1},...,r_{n})$. Taken together, they express that, at time $\alpha^{+i}$, $P(\rho)$ assumes the halting configuration $c$. 

Thus $\alpha^{+i}$ is minimal with the property that $L_{\alpha^{+i}}\models\phi(\rho,c)$. Now, by Lemma \ref{gostanian}, this implies that $\alpha^{+i}$ is Gandy. By Lemma \ref{gandy and beta}, we have $\beta(\alpha)>(\alpha^{+i})^{+}=\alpha^{+(i+1)}$, as desired. 
\end{proof}

We do not know whether the first alternative can occur for any countable $\alpha$ unless $L_{\alpha}\models$ZF$^{-}$. 

The argument for Theorem \ref{weak or very strong} actually shows that, when $\beta(\alpha)\geq\alpha^{+}$, then $\beta(\alpha)$ cannot lie between two successive admissible ordinals.  By iterating the same argument, we obtain:

\begin{corollary}
	Let $\alpha$ be admissible and countable in $L_{\alpha^{+}}$. Suppose that $\beta(\alpha)\geq\alpha^{+}$. Then $\beta(\alpha)\in[\delta,\delta^{+})$, where $\delta$ is a limit of admissible ordinals $>\alpha$. 
\end{corollary}

\section{Weak ITRMs and $u$-weakness}

Concerning $\alpha$-ITRMs, one of our main results above is that their halting times are bounded by $\alpha^{\omega}$ if and only if $L_{\alpha}\models\text{ZF}^{-}$. Moreover, in Theorem \ref{pi3 reflecting} above, we saw that the halting times of $\alpha$-wITRMs are bounded by $\alpha$ itself when $\alpha$ is $\Pi_{3}$-reflecting. 
This motivates the following definition:

\begin{defini}{\label{def u-weak}}
An ordinal $\alpha$ is $u$-weak if and only if all $\alpha$-wITRM-clockable ordinals are smaller than $\alpha$.
\end{defini}

Clearly, since we allow parameters, all ordinals $<\alpha$ are $\alpha$-wITRM-clockable, so that $\alpha$ coincides with the supremum of the $\alpha$-wITRM-clockable ordinals when $\alpha$ is $u$-weak. Also note that, by the speedup-theorem, an ordinal $\alpha$ is $u$-weak if and only if $\alpha$ is not $\alpha$-wITRM-clockable. 

We currently have no full characterization for $u$-weakness. By the proof of Theorem \ref{pi3 reflecting}, all $\Pi_{3}$-reflecting ordinals are $u$-weak. In this section, we will 
additionally prove the following:

\begin{itemize}
\item An admissible ordinal $\alpha$ is $u$-weak if and only if it is wITRM-regular (i.e. not wITRM-singular).
\item Any $u$-weak ordinal is admissible.
\item There are $u$-weak ordinals that are not $\Pi_{3}$-reflecting.
\item There are admissible ordinals that are not $u$-weak.
\end{itemize}

Thus, $u$-weakness is strictly between $\Pi_3$-reflection and admissibility. We now prove the statements in the order of their appearence.

\begin{thm}{\label{u-weak regular}}
An admissible ordinal $\alpha$ is $u$-weak if and only if it is wITRM-regular.
\end{thm}
\begin{proof}

We will prove both implications by contraposition. 

``$\Rightarrow$'': Suppose that $\alpha$ is not wITRM-regular. Thus, there is an $\alpha$-wITRM-computable, total and cofinal function $f:\beta\rightarrow\alpha$ with $\beta<\alpha$. Let $P$ be a program that computes $f$, say in the parameter $\vec{p}$. If $P(\iota,\vec{p})$ halts after $\geq\alpha$ many steps on some input $\iota<\beta$, then we have found an $\alpha$-wITRM-program that halts in $\geq\alpha$ many steps, so $\alpha$ is not $u$-weak. On the other hand, suppose that $P(\iota,\vec{p})$ takes $<\alpha$ many steps on any input $\iota<\beta$. In this case, we add the parameter $\beta$ to our computation and use a separate register $R$. In $R$, we count upwards from $0$ to $\beta$ and for every $\iota<\beta$, we use $P$ to compute $f(\iota)$ and then use another extra register to count from $0$ to $f(\iota)$. When the content of $R$ reaches $\beta$, we halt. Clearly, this program halts after at least $\alpha$ many steps, so again, $\alpha$ is not $u$-weak. 

\medskip

``$\Leftarrow$'': Now suppose that $\alpha$ is not $u$-weak. As we mentioned above, this implies that $\alpha$ is $\alpha$-wITRM-clockable. In particular, there is a program $P$ that halts in $\alpha$ many steps (we ignore parameters for the sake of simplicity). We distinguish two cases:

\medskip
\textbf{Case $1$}: For every register $R$ used by $P$, there is a time $\tau<\alpha$ such that, from time $\tau$ on, the content of $R$ never dropped below the content of $R$ at time $\alpha$. 

By the liminf-rule, there must also be an ordinal $\rho<\alpha$ such that the active program line index was never below the one at time $\alpha$ after time $\rho$. Let $\mu$ be the maximum of $\rho$ and the finitely many $\tau$ that exist by the case assumption. 
Again by the liminf-rule, the active program line and the register contents at time $\alpha$ must have occured cofinally often before time $\alpha$. Thus, we can build an interleaving, strictly increasing sequence of length $\omega$ of times at which these values were the `right' ones. By admissibility, the supremum $\bar{\alpha}$ of this sequence will be $<\alpha$. But then, $\bar{\alpha}$ and $\alpha$ witness that $P$ is looping, which is a contradiction.

\medskip
\textbf{Case $2$}: There is some register $R$ containing an ordinal $\rho$ at time $\alpha$ such that the content of $R$ was $<\rho$ cofinally often before time $\alpha$. 

By the liminf-rule, this means that, for every $\iota<\rho$, there must be some $\xi<\alpha$ such that, from time $\xi$ on, the content of $R$ was $\geq\xi$. Let $g$ be the function that maps each $\iota<\rho$ to the minimal such $\xi$. Clearly, $g$ maps $\rho$ cofinally into $\alpha$. (If $g[\rho]$ was bounded in by $\beta<\alpha$, all contents of $R$ would be $\geq\rho$ from time $\beta$ on, contradicting the case assumption.) We claim that $g$ is $\alpha$-wITRM-computable. To this end, we use the clockability of $\alpha$. 
Reserve two extra registers, say $T_{1}$ and $T_{2}$. Now, given $\iota<\rho$, we proceed as follows: In $T_{1}$, we count upwards, starting with $0$. For every value $\zeta\in T_{1}$, we run $P$ for $\alpha$ many steps, using the clockability of $\alpha$ and check whether, from time $\zeta$ on, the content of $R$ drops below $\iota$. If yes, we continue with the next value of $\iota$. If not, we halt with output $\zeta$. Since we know that some $\zeta<\alpha$ exists for which the routine will halt, this computes $g(\iota)$ without producing an overflow and thus on an $\alpha$-wITRM. Thus, $\alpha$ is wITRM-singular, i.e. not wITRM-regular.

\end{proof}

\begin{remark}
Note that only the reverse direction uses the admissibility of $\alpha$, which can in fact be weakened to the assumption that $\Sigma_{1}$-definable total functions with domain $\omega$ are bounded. 
\end{remark}

\begin{thm}{\label{u-weak admissible}}
Any $u$-weak ordinal is admissible.
\end{thm}
\begin{proof}
Suppose for a contradiction that $\alpha$ is $u$-weak, but not admissible. Thus, there is $\beta<\alpha$ and a cofinal function $f:\beta\rightarrow\alpha$ which is $\Sigma_{1}$ over $L_{\alpha}$. Let $g$ be the function that maps $\iota<\beta$ to the smallest $\xi<\alpha$ such that $L_{\xi}$ believes that $f(\iota)$ exists, according to the $\Sigma_{1}$-definition of $f$. Clearly, $g$ is also a cofinal map from $\rho$ to $\alpha$. Now, given $\iota<\beta$, compute upwards in a separate register $T$ starting with $\iota$ and, for every content $\zeta$, use evaluation of bounded truth predicates to test whether $L_{\zeta}\models\exists{\delta}f(\iota)=\delta$. 
If not, continue with the next value of $\zeta$. Otherwise, halt with output $\zeta$. Since this will halt for some $\zeta<\alpha$, this will compute $g(\iota)$ without producing an overflow, and thus by an $\alpha$-wITRM-computation. 
Hence, $\alpha$ is wITRM-singular and hence not $u$-weak by the remark after Theorem \ref{u-weak regular}, a contradiction. 
\end{proof}

\begin{thm}{\label{u-weak not necessary or sufficient}}

(1) There are $u$-weak ordinals that are not $\Pi_{3}$-reflecting. In fact, there are unboundedly many such ordinals. 

(2) There are admissible ordinals that are not $u$-weak.
\end{thm}
\begin{proof}
(1) We claim that there is a $\Pi_{3}$-sentence $\phi$ such that $L_{\alpha}\models\phi$ if and only if $\alpha$ is $u$-weak. Once this is proved, consider some $u$-weak ordinal $\mu$ and let $\beta$ be the first $\Pi_{3}$-reflecting ordinal $>\mu$. Then, by reflection, there is $\gamma\in(\mu,\beta)$ such that 
$L_{\gamma}\models\phi$, and so $\gamma$ is $u$-weak and $>\mu$, but not $\Pi_{3}$-reflecting. Since all $\Pi_{3}$-reflecting ordinals are $u$-weak, there are unboundedly many $u$-weak ordinals, and the claim follows. 

Now for the claim: $\alpha$ is $u$-weak if and only if, for all programs $P$ and all parameters $\vec{p}\subseteq\alpha$, one of the following holds in $L_{\alpha}$:

\begin{itemize}
\item There is an overflow at time $\alpha$, i.e. $\forall{\iota}\exists{\tau}\exists{i\in\omega}\forall{\xi>\tau}R_{i\xi}>\iota$ (where $i$ denotes the register with index $i$ and $R_{i\xi}$ is the content of $R_{i}$ at time $\xi$ in the computation of $P(\vec{p})$) (this is a $\Pi_{3}$-condition) OR
\item $P(\vec{p})$ halts (this is a $\Sigma_{1}$-condition) OR
\item $P(\vec{p})$ does not halt, i.e. there are $\tau_{1}<\tau_{2}$ such that $P(\vec{p})$ is in a strong loop between times $\tau_1$ and $\tau_2$ (this is a $\Sigma_{1}$-condition).
\end{itemize}

This is clearly necessary for $\alpha$ being $u$-weak. To see that it is sufficient, note that every program $P$ only uses finitely many registers; thus, if the first disjunct holds, one of them has contents that eventually surpass cofinally many ordinals below $\alpha$, which suffices for an overflow.

The disjunction can clearly be written as a $\Pi_{3}$-formula, and thus the same holds for the whole condition. 

\bigskip

(2) Since $\omega_{1}^{\text{CK}}$-wITRMs (in fact, $(\omega+1)$-wITRMs) can simulate ($\omega$-)ITRMs whose running times are all ordinals below $\omega_{\omega}^{\text{CK}}$ (see \cite{K1}), it follows that $\omega_{1}^{\text{CK}}$ is admissible, but not $u$-weak. The same holds for 
$\omega_{i}^{\text{CK}}$ for all $i\in\omega$.

\end{proof}




%

\begin{prop}
	There is an $\alpha$-wITRM-program $P_{\text{WO}}$ such that, for any $c\subseteq\beta<\alpha$, $P_{\text{WO}}^{c}\downarrow=1$ if and only if $c$ codes a well-ordering and otherwise, $P^{c}\downarrow=0$. 
\end{prop}
\begin{proof}
	Just perform the usual depth-first search used on ITRMs (see \cite{KM}) with $\beta$ at the bottom of the stack. 
\end{proof}

As we saw for $\beta(\alpha)$, we can now see that $u$-weak ordinals are WO-true. 

\begin{lemma}{\label{u-weak WO-true}}
	If an ordinal $\alpha$ is $u$-weak, then $\alpha$ is admissible and WO-true.
\end{lemma}
\begin{proof}
 Suppose that $\alpha$ is $u$-weak. By Theorem \ref{u-weak admissible}, $\alpha$ is admissible. Let $R\in L_{\alpha}$ be a linear ordering and suppose that $R$ is not well-founded. Since $R\in L_{\alpha}$, some $\beta$-code $c$ for $R$ is $\alpha$-wITRM-computable for some $\beta<\alpha$. By $u$-weakness, $P_{\text{WO}}^{c}$ will terminate in $\gamma<\alpha$ many steps after finding an ill-founded sequence in $R$. Now, the computation $D$ of $P_{\text{WO}}^{c}$ is contained in $L_{\alpha}$ and from $D$, one can define an ill-founded sequence $(a_{i}:i\in\omega)$ in $R$ by letting $a_{i}=\delta$ if and only if, for some $\iota<\gamma$, the $i$th component of content of the stack register is always $\delta$. But then, we have $(a_{i}:i\in\omega)\in L_{\alpha}$, as desired. 
\end{proof}

\subsection{Further Observations}

We mention a bunch of related results, some of which are contained in \cite{C} as exercises, in the hope that these may lead to more substantial generalizations or refinements.

\begin{prop}{\label{strong witrms}}
The real numbers computable by an $\alpha$-wITRM for $\alpha>\omega$ are a superset of $\mathfrak{P}(\omega)\cap L_{\omega_{\omega}^{\text{CK}}}$.
\end{prop}
\begin{proof}
It is easy to simulate ITRMs on an $\alpha$-wITRM when $\alpha>\omega$.
\end{proof}

We also note the following humble simulation result:

\begin{prop}{\label{beyond omega}}
For any $1<k\in\omega$ and any ordinal $\alpha$, we have $\beta(\alpha+1)=\beta(\alpha\cdot k)$.
\end{prop}
\begin{proof}
It is clear that we have $\beta(\alpha+1)\leq\beta(\alpha\cdot k)$
	
On the other hand, an $(\alpha\cdot k)$-ITRM-program $P$ can be simulated on an $(\alpha+1)$-ITRM in the following way: Replace any register used by $P$ with $k$ registers that can contain ordinals up to $\alpha$. Then represent $(\alpha\cdot i)+\rho$, $i<k$, $\rho<\alpha$ by having $\alpha$ in the first $i$ many of these registers and $\rho$ in the $(i+1)$st.
\end{proof}

Finally, we observe that the lost melody theorem holds for $\alpha$-ITRMs whenever $\alpha$ is exponentially closed. Let us say that $x\subseteq\alpha$ is $\alpha$-ITRM-recognizable when there are an $\alpha$-ITRM-program $P$ and an ordinal $\zeta<\alpha$ such that, for any $y\subseteq\alpha$, we have $P^{y}(\zeta)\downarrow=1$ if and only if $y=x$ and otherwise $P^{y}(\zeta)\downarrow=0$.\footnote{The term `recognizable' was first used by Hamkins and Lewis in \cite{HL} in the context of ITTMs.} Following the terminology of \cite{HL} (where it is shown that there are lost melodies for ITTMs), a subset of $\alpha$ which is $\alpha$-ITRM-recognizable, but not $\alpha$-ITRM-computable is called an $\alpha$-ITRM lost melody, below simply called `lost melody' for short. In the below proof, we will occasionally confuse a program $P$ with its index $i$.

\begin{thm}{\label{alpha  ITRM lost melody}}
For any ITRM-singular and exponentially closed ordinal $\alpha$, there is a lost melody. 
\end{thm}
\begin{proof}
Given $\alpha$, let $H:=\{p(i,\zeta)\in\omega\times\alpha:P_{i}(\zeta)\downarrow\}$ be the halting set for $\alpha$-ITRMs. It is clear that $H$ is not $\alpha$-ITRM-computable. 

We claim that $H$ is $\alpha$-ITRM-recognizable. We sketch the proof, which is a generalization of the one in \cite{C2} showing that the halting set for ITRMs is ITRM-recognizable and thus a lost melody for ITRMs. We freely use the relativized versions of the previous results in this section: In particular, if $\gamma$ is $\alpha$-ITRM-clockable in the oracle $x\subseteq\alpha$, then $\gamma$ is $\alpha$-ITRM-computable in this oracle etc.

Note that it is easy to effectively assign to any $\alpha$-ITRM-program $P$ and any parameter $\zeta$ another $\alpha$-ITRM-program $Q^{P,\zeta}$ such that $Q$ halts if and only if $P(\iota,\zeta)$ halts with output $0$ or $1$ for any $\iota<\alpha$.
 Moreover, we can effectively assign to all $\alpha$-ITRM-programs $P$ and all ordinals $\iota,\zeta<\alpha$ a program $R^{P,\iota,\zeta}$ that halts if and only if $P(\iota)\downarrow=\zeta$. 

Now, given $x\subseteq\alpha$ in the oracle, we do the following: Run through $\alpha$ and, for each $p(i,\zeta)<\alpha$, check whether $p(Q^{P_{i},\zeta},\zeta)\in x$. If not, continue. Otherwise, perform a well-foundedness check on the set $\{\iota<\alpha:\{p(R^{P_{i},(\iota,\zeta),1},p(\iota,\zeta))\in x\}$. Doing this for any $(i,\zeta)\in\omega\times\alpha$ will clock some ordinal $\gamma$ in the oracle $x$; thus, $\gamma$ is $\alpha$-ITRM-computable. From a code $c$ for $\gamma$, we can then compute a code $d$ for $L_{\gamma}$. Evaluating truth in $L_{\gamma}$, we can then check whether it holds in $L_{\gamma}$ that, for any $\alpha$-ITRM-program $P$ and any $\zeta<\alpha$, $P(\zeta)$ either halts or runs into a strong loop. If this fails, then $x\neq H$. If this holds, $H$ is definable over $L_{\gamma}$ and thus $\alpha$-ITRM-computable from $d$, and we can use this to compute $H$ and compare it to $x$.
\end{proof}


\section{Conclusion and Further Work}

The above work settles the question of the computational strength of $\alpha$-ITRMs when $L_{\alpha}\models$ZF$^{-}$; in the other cases, the question for the $\alpha$-ITRM-computable subsets of an exponentially closed ordinal $\alpha$ is reduced to the determination of $\beta(\alpha)$, which is characterized as (i) the supremum of the $\alpha$-ITRM-clockable ordinals (ii) the $\alpha$-ITRM-computable ordinals and (iii) the supremum of the looping times for non-halting $\alpha$-ITRM-programs. Although we obtained some lower and upper bounds, these are still quite far apart, and we expect that considerably new ideas are needed to determine $\beta(\alpha)$ precisely for any $\alpha$ which neither has $L_{\alpha}\models$ZF$^{-}$ nor $\alpha=\omega$. 


Similarly open is the analogous question for $\alpha$-wITRMs: Here, we are even missing a characterization of the $u$-weak ordinals. 

We mention the following specific questions:

\begin{question}
	Is $\beta(\alpha)\leq\alpha^{+\omega}$ for all exponentially closed $\alpha$? Even more boldly, is $\beta(\alpha)=\alpha^{+\omega}$ for such $\alpha$ unless $L_{\alpha}\models$ZF$^{-}$?
\end{question}

\begin{question}
An important feature of ITRMs is the solvability of the bounded halting problem, see Koepke and Miller \cite{KM}: For any fixed number $k\in\omega$, the halting problem for ITRM-programs using $k$ registers is solvable by an ITRM-program (which, of course, will use more than $k$ registers). The proof in \cite{KM} seems to depend on the fact that, for any possible register content $j$ of an ITRM, only finitely many configurations have all register contents $\leq j$, which clearly fails for $\alpha$-ITRMs as soon as $\alpha>\omega$. Hence, we ask: Does the solvability of the bounded halting problem work for any exponentially closed $\alpha>\omega$ other than the ZF$^{-}$-ordinals?\footnote{For the ZF$^{-}$-ordinals, this is clearly true as we can clock the upper bound $\alpha^{n+1}$ of the halting times of programs using $\leq n$ registers, so it can be decided whether such a program halts by simply running it for that many steps and seeing whether it holds until then.}
\end{question}

\section{Acknowledgements}

We thank Philipp Schlicht for a series of discussion in which the proof of a previous (weaker) version Lemma $5$, was obtained and his kind permission to use this proof in our work.


\begin{thebibliography}{}
\bibitem[AS]{AS} F. Abramson, G. Sacks. Uncountable Gandy Ordinals. Journal of the London Mathematical Society, vol. s2-14(3) (1976)
    \bibitem[C]{C} M. Carl. Ordinal Computability. An Introduction to Infinitary Machines. De Gruyter (2019) (forthcoming)
    \bibitem[CFKMNW]{CFKMNW} M. Carl, T. Fischbach, P, Koepke, R. Miller, M. Nasfi, G. Weckbecker. The basic theory of Infinite Time Register Machines. Archive for Mathematical Logic 49 (2010) 2, 249-273. 
    \bibitem[C1]{alpha itrms} M. Carl. Resetting $\alpha$-register machines and ZF$^{-}$. Preprint, arxiv 1907.09513 (2019)
   \bibitem[C2]{C2} M. Carl. Optimal results on recognizability for infinite time register machines. ournal of Symbolic Logic 80 (4):1116-1130 (2015) 
 \bibitem[Cu]{Cu} N. Cutland. Computability. An introduction to recursive function theory. Cambridge University Press (1980)
    \bibitem[FW]{FW} S. Friedman, P. Welch. Hypermachine. J. Symb. Log., vol. 76(2), pp. 620-636 (2011)
\bibitem[GJH]{GJH} V. Gitman, T. Johnstone, J. Hamkins. What is the theory ZFC without power set? Mathematical Logic Quarterly, vol. 62(4) (2011)
\bibitem[Go]{Go} R. Gostanian. The next admissible ordinal. Annals of Mathematical Logic, vol 17(1-2) (1979)
    \bibitem[HL]{HL} J. Hamkins, A. Lewis. Infinite Time Turing Machines.  Journal of Symbolic Logic 65(2), 567--604 (2000)
\bibitem[HMO]{HMO} J. Hamkins. MathOverflow Post, \url{https://mathoverflow.net/questions/345007/relation-between-eta-and-omegal-1/345037#345037}, accessed: 11.11.2019
	\bibitem[KS]{KS} P. Koepke, B. Seyfferth. Ordinal machines and admissible recursion theory. Annals of Pure and Applied Logic, vol. 160, pp. 310--318 (2009)
	\bibitem[KS1]{KS1} P. Koepke, B. Seyfferth. Towards a theory of infinite time Blum-Shub-Smale machines. 
	S. Cooper et al. (eds.), How the world computes. Turing centenary conference and 8th conference on computability in Europe, CiE 2012, 
	Cambridge, UK, 2012. Proceedings. Springer Berlin. Lecture Notes in Computer Science 7318, pp. 405--415 (2012). 
	\bibitem[KM]{KM} P. Koepke, A. Morozov. On the computational strength of Infinite Time Blum-Shub-Smale Machines. Algebra and Logic, vol. 56, no. 1, (2017)
  \bibitem[K]{K} P. Koepke. Infinite time register machines. In Logical Approaches to Computational Barriers, Arnold Beckmann et al., eds., Lecture Notes in Computer Science 3988 (2006), 257-266
	\bibitem[K1]{K1} P. Koepke. Ordinal Computability. In Mathematical Theory and Computational Practice. K. Ambos-Spies et al. (eds.), 
Lecture Notes in Computer Science 5635, pp. 280--289 (2009)
\bibitem[KM]{KM} P. Koepke, R. Miller. An enhanced theory of infinite time register machines. In Logic and Theory of Algorithms. A. Beckmann et al, eds., Lecture Notes in Computer Science 5028 (2008), 306-315
\bibitem[M]{M} D. Madore. A Zoo of ordinals. Available online. \url{http://www.madore.org/~david/math/ordinal-zoo.pdf}
\bibitem[OTM]{OTM} P. Koepke. Turing Computations on Ordinals. Bull. of Symbolic Logic, Volume 11(3) (2005)
\bibitem[ORM]{ORM} P. Koepke, R. Siders. Register computations on ordinals. Archive for Mathematical Logic vol. 47, pp. 529--548 (2008)
\bibitem[W]{W} P. Welch. Characteristics of discrete transfinite Turing machine models: halting times, stabilization times, and Normal Form Theorems. Theoretical Computer Science, vol. 410, (2009), 426-442
\bibitem[W1]{W1} P. Welch. Transfinite Machine Models. In: R. Downey (ed.), Turing's Legacy, Lecture Notes in Logic, Association for Symbolic Logic, (2013)
\end{thebibliography}
\end{document}